\documentclass{amsart}
\usepackage{amssymb,amscd,amsthm,amsmath,amsfonts}

\newcommand{\norm}[1]{\mbox{$\left\|#1\right\|$}}
\newcommand{\x}{\times}
\newcommand{\cs}{\mbox{$C^{*}$-algebra}}
\newcommand{\css}{\mbox{$C^{*}$-algebras}}

\newcommand{\C}{\mathbb{C}}

\newcommand{\R}{\mathbb{R}}

\newcommand{\ov}[1]{\mbox{$\overline{#1}$}}

\newcommand{\al}{\mbox{$\alpha$}}
\newcommand{\eps}{\mbox{$\epsilon$}}
\newcommand{\bt}{\mbox{$\beta$}}
\newcommand{\ga}{\mbox{$\gamma$}}
\newcommand{\Ga}{\mbox{$\Gamma$}}
\newcommand{\de}{\mbox{$\delta$}}
\newcommand{\De}{\mbox{$\Delta$}}

\newcommand{\la}{\mbox{$\lambda$}}
\newcommand{\ka}{\mbox{$\kappa$}}
\newcommand{\La}{\mbox{$\Lambda$}}

\newcommand{\si}{\mbox{$\sigma$}}

\newcommand{\mfA}{\mathfrak{A}}

\newcommand{\mfg}{\mathfrak{g}}

\newcommand{\mfH}{\mathfrak{H}}
\newcommand{\mfh}{\mathfrak{h}}
\newcommand{\mfK}{\mathfrak{K}}
\newcommand{\mfL}{\mathfrak{L}}

\newcommand{\bgc}{\begin{center}}

\newcommand{\edc}{\end{center}}
\newcommand{\be}{\begin{enumerate}}
\newcommand{\ee}{\end{enumerate}}
\newcommand{\beq}{\begin{equation}}
\newcommand{\eeq}{\end{equation}}
\newcommand{\beqn}{\begin{eqnarray}}
\newcommand{\eeqn}{\end{eqnarray}}
\newcommand{\beqns}{\begin{eqnarray*}}
\newcommand{\eeqns}{\end{eqnarray*}}
\newcommand{\bq}{\begin{quote}}
\newcommand{\eq}{\end{quote}}

\newcommand{\bi}{\begin{itemize}}
\newcommand{\ei}{\end{itemize}}
\newcommand{\bd}{\begin{description}}
\newcommand{\ed}{\end{description}}

\newcommand{\lan}{\mbox{$\langle$}}
\newcommand{\ran}{\mbox{$\rangle$}}

\theoremstyle{plain}
\newtheorem{theorem}{Theorem}
\newtheorem{lemma}{Lemma}
\newtheorem{definition}{Definition}
\newtheorem{proposition}{Proposition}
\newtheorem{corollary}{Corollary}

\numberwithin{equation}{section}

\begin{document}
\title{The equivariant analytic index for proper groupoid actions}
\author{Alan L. T. Paterson}
\address{Department of Mathematics\\
University of Mississippi\\
University, MS 38677}
\email{mmap@olemiss.edu}
\date{}
\begin{abstract}
The paper constructs the analytic index for an elliptic pseudodifferential family of 
$L^{m}_{\rho,\de}$-operators invariant under the proper action of a continuous family groupoid on a $G$-compact, $C^{\infty,0}$ $G$-space.
\end{abstract}

\maketitle

\section{Introduction}
In his
book (\cite[p.151]{Connesbook}), Alain Connes points out that a number of index theorems 
can all be formulated within the following general framework which he calls {\em the general conjecture for smooth groupoids}.  There are given a Lie groupoid $G$ acting properly on a manifold $X$, and a $G$-invariant, elliptic family $D$ of pseudodifferential operators determining a map between the compactly supported, smooth sections of a pair of $G$-vector bundles $E,F$ over $X$.  One then has to prove the index theorem in this context. 

In this paper we will extend and establish part of this program.  Following the classical papers of Atiyah and Singer ((\cite{AS1,AS3,AS4})), the proof of such a theorem falls into four primary parts.  First we have to construct the analytic index in an appropriate $K$-group.  (This $K$-group will be $K_{0}(C_{red}^{*}(G))$.)  Next, we have to construct the topological index in the same $K$-group.  Then we have to show that these $K$-elements are the same.  Finally, a cohomological formula for the index - corresponding to that in classical index theory (\cite{AS3}) - is to be obtained.  Examples of such a formula have been obtained by Connes and Moscovici in the case of the $L^{2}$-index theorem for  homogeneous spaces of Lie groups
(\cite{CoMos})  and by Connes and Skandalis for the longitudinal index theorem (\cite{CoSkand}). 

The objective of this paper is to give a reasonably detailed proof of the construction of the first of these parts, i.e. 
that of the analytic index.  We will in fact work in the more general continuous family context. A number of papers are relevant to this construction in the context where $X=G$, e.g. \cite{LMN00,LMN01,LN,Mont,Mont3,Montthesis,MontP,Nistor,NWX}.  
A version of our construction under special assumptions - including the assumptions that $G^{0}$ is a proper, $G$-compact, $G$-space and that $X$ is a fiber bundle over $G^{0}$ with compact smooth manifold as fiber -  is proved in \cite{jsrc}.  In that paper, the author asserted his belief that the proof could be adapted to give the analytic index in complete generality.  The present paper carries out this general proof in the continuous family context.  Continuous family groupoids, in the case of {\em holonomy groupoids}, were effectively considered by Connes in his paper on non-commutative integration (\cite[p.112]{Cointeg}) where they arise from  $C^{\infty,0}$ foliations.  We will use his $C^{\infty,0}$ notation throughout the paper.  The reader is referred to \S 6 for motivating examples for the theory presented here.

Sections 2 and 3 of the paper recapitulate and extend the theory of continuous families and 
continuous family groupoids discussed in \cite{families}.  The theory, which of course parallels the corresponding $C^{\infty}$-theory, seems to require a reasonably detailed description.  We also incorporate in \S 3 the $C^{\infty,0}$-versions of the results on proper Lie groupoid actions in \cite{jsrc}.
The main improvement in the theory of \S 3 is that it is shown that continuous family groupoids 
(and indeed, proper $C^{\infty,0}$ $G$-spaces in general)
always have $C^{\infty,0}$ left Haar systems.  (In \cite{families} it was shown that only {\em continuous} left Haar systems exist.)  A $C^{\infty,0}$ left Haar system is required for our analytic index theorem.

Following the sketch of a group equivariant index theorem by Kasparov (\cite{Kasp1}), rather than restricting to the classical pdo's used by Atiyah and Singer, we use (the more general) $L^{m}_{\rho,\de}$-operators of H\"{o}rmander.  Also, for the purpose of the paper, we need to develop 
a theory of ``locally continuously varying'' families of such pdo's.  The author has been unable to find this theory in the literature and so it is developed within the paper (in \S 4).  The proofs are not much more involved than those in the classical case.  Also, the results that we need for continuously varying families of pdo's generalize fairly routinely from the single operator case given in the book by Shubin (\cite{Shubin}).  
 
Next, in \cite{jsrc}, the analytic index was obtained by using the equivariant K-theory of N. C. Phillips (\cite{Phillips}), generalized to groupoid actions.  In the present case, we have been able to avoid the use of Phillips's theory by going directly to a Kasparov module.  A key technique of the theory presented here is a groupoid version of the averaging technique of Connes and Moscovici, where they use a cut-off function in the proof of the homogeneous index theorem.  Indeed, the Kasparov module that we use to obtain the analytic index can be regarded as the  ``averaged'' natural Kasparov module in the non-equivariant case for compactly supported elliptic pdo families.  Comments pertinent to and examples illustrating the theory developed here are discussed in the final section of the paper.  Part (5) of the final section contains a very brief description of an approach to the topological index side of Connes's conjecture.  Lastly, I wish to express my gratitude to the referee, whose many comments and suggestions have greatly improved this paper.

\section{Continuous families of manifolds}

Continuous families of compact manifolds were introduced by Atiyah and Singer
(\cite{AS4}) in connection with their index theorem for families. They are fiber bundles with smooth compact fiber.  The more
general notion of a continuous family of manifolds that will be needed in our
discussion of the analytic index is developed in \cite{families} and the reader
is referred to that paper for more details pertinent to the discussion below.

All locally compact spaces are assumed to be second countable and Hausdorff.
If $X$ is a locally compact space and $M$ is a (smooth) manifold, then
$C(X)$, $C^{\infty}(M)$ are respectively the spaces of continuous and smooth
complex-valued functions on $X,M$.  The subspaces of $C(X), C^{\infty}(M)$
consisting of the functions in $C(X), C^{\infty}(M)$ with compact support are
denoted by $C_{c}(X), C_{c}^{\infty}(M)$.  The family of compact subsets of $X$ is denoted by 
$\mathcal{C}(X)$.

Let $T$ be a locally compact space
and $M, N$ be manifolds.  Then (cf. \cite[p.110]{Cointeg} for the foliation case) a function $f:T\x M\to T\x N$ is said to be a {\em $C^{\infty,0}$-function} ($f\in C^{\infty,0}(T\x M,T\x N)$) if for all $t\in T$, 
$f(\{t\}\x M)\subset \{t\}\x N\cong N$ and the map 
$t\to f^{t}$, where $f^{t}$ is the restriction of $f$ to $\{t\}\x M$, is continuous from 
$T$ into $C^{\infty}(M,N)$ (with the usual topology of uniform convergence on compact sets of functions for all derivatives).  The $C^{\infty,0}$ notion extends in the obvious ``local'' way to continuous functions $f:U\to V$, where $U, V$ are open subsets of $T\x M, T\x N$ for which 
$p_{1}(U)=p_{1}(V)$, where   $p_{1}$ is the projection onto the first coordinate.    

Now let $X$ be a locally compact space and $p:X\to T$ be a continuous, open surjection.  For each $t$ we put $X^{t}=p^{-1}(\{t\})$.

\begin{definition}             \label{def:cf}
The pair $(X,p)$
is defined to be {\em a continuous family of manifolds over $T$} if
there exists $k\geq 1$ and a set of pairs $\{(U_{\al},\phi_{\al}): \al\in A\}$, where each
$U_{\al}$ is an open subset of $X$ and $\cup_{\al\in A}U_{\al}=X$, such that:
\bi
\item[(i)] for each $\al$, the map
$\phi_{\al}$ is a fiber preserving homeomorphism from $U_{\al}$ onto $p(U_{\al})\x V_{\al}$
where $V_{\al}$ is an open subset of $\R^{k}$;
\item[(ii)] for each $\al, \bt$, the mapping 
$\phi_{\bt}\circ\phi_{\al}^{-1}:\phi_{\al}(U_{\al}\cap U_{\bt})\to
\phi_{\bt}(U_{\al}\cap U_{\bt})$ is $C^{\infty,0}$.
\ei
\end{definition}

As in the case of manifolds,
the family $\mfA=\{(U_{\al},\phi_{\al}):\al\in A\}$ is called
an {\em atlas} for the continuous family $(X,p)$, and the
$(U_{\al},\phi_{\al})$'s, are called {\em charts}.   For each $t\in T$, the family
$\{U_{\al}\cap X^{t}:\al\in A\}$ is an atlas for a $k$-dimensional manifold structure on $X^{t}$ giving the relative topology of $X^{t}$ as a subset of $X$.

We will take the atlas $\mfA$ to be
maximal.  Then $\mfA$ is a basis for the topology of $X$.  For such a chart $(U,\phi)$, we will write $U\thicksim \phi(U)=p(U)\x W$. 
The simplest example of a continuous family over $T$ is one of the form
$X=T\x M$ where $M$ is a manifold. Such a family is called {\em trivial}.
Every continuous family is locally trivial.  Every open subset $U$ of a continuous family $X$ is itself a continuous family over $p(U)$.

Now let $(X_{1},p_{1}), (X_{2},p_{2})$ be continuous families
over $T$. Their {\em fibered product} $(X_{1}*X_{2},p)$ over $T$ is defined:
\[ X_{1}*X_{2}=\{(x_{1},x_{2})\in X_{1}\x X_{2}: p_{1}(x_{1})
=p_{2}(x_{2})\} \]
and $p(x_{1},x_{2})=p_{1}(x_{1})=p_{2}(x_{2})$. With the relative topology 
on $X_{1}*X_{2}$ as a subset of $X_{1}\x X_{2}$, the pair $(X_{1}*X_{2},p)$ is
a continuous family of manifolds over $T$ in the natural way, with fibers 
$(X_{1}*X_{2})^{t}={X_{1}}^{t}\x {X_{2}}^{t}$ having the product manifold
structure.  Charts for the $C^{\infty,0}$-structure on $X_{1}*X_{2}$ are given by sets $U_{1}*U_{2}$, where $U_{i}\thicksim 
p_{i}(U_{i})\x W_{i}$, $p_{1}(U_{1})=p_{2}(U_{2})$, and in the obvious way, 
$U_{1}*U_{2}\thicksim p(U_{1}*U_{2})\x (W_{1}\x W_{2})$.  

Another way in which $X_{1}*X_{2}$ is a continuous family of manifolds is as follows.  It is easy to see that $(X_{1}*X_{2},t_{1})$ is such a family over $X_{1}$, where 
$t_{1}(x_{1},x_{2})=x_{1}$.

Pull-backs of continuous families are themselves continuous families.
Specifically, let $(X,p)$ be a continuous family over $T$, $Z$ be a locally
compact Hausdorff space and $t:Z\to T$ be a continuous map.  The pull-back
continuous family $(t^{*}X,p')$ over $Z$ is given by: $t^{*}X=\{(z,x)\in Z\x X: t(z)=p(x)\}$,               
and the map $p'$ by: $p'((z,x))=z$.  If $U\thicksim p(U)\x W$ is a chart for $X$, then 
$\{(z,x):t(z)=p(x), x\in U\}\thicksim t^{-1}(p(U))\x W$ is a chart for $t^{*}X$.
The continuous family $(X_{1}*X_{2},t_{1})$ is the pull-back family $(p_{1}^{*}X_{2},p')$. 

We need the notion of a {\em morphism} of continuous families.  
Let $(X_{1},p_{1}), (X_{2},p_{2})$ be continuous families over $T_{1}, T_{2}$.
Let $q:T_{1}\to T_{2}$ be a continuous
map and $f:X_{1}\to X_{2}$ be a continuous fiber preserving map with respect
to $q$ in the sense that $p_{2}\circ f=q\circ p_{1}$. Let $(U_{1},\phi_{1}), (U_{2},\phi_{2})$ be charts in $X_{1},X_{2}$ with $\phi_{i}(U_{i})=p_{i}(U_{i})\x W_{i}$ and 
$f(U_{1})\subset U_{2}$.   So $q(p_{1}(U_{1}))\subset p_{2}(U_{2})$ and 
$f(X_{1}^{t})\subset X_{2}^{q(t)}$.  Write 
$\phi_{2}f\phi_{1}^{-1}(t,w)=(q(t),f^{t}(w))$.  
Then the pair $(f,q)$ is called a {\em morphism} or a {\em $C^{\infty,0}$ map} if the map $t\to f^{t}$ is continuous from $p_{1}(U_{1})$ into $C^{\infty}(W_{1},W_{2})$.  Obviously, in that case, for each $t$, the map $x\to f(x)$ is $C^{\infty}$ from $X_{1}^{t}$ into $X_{2}^{q(t)}$.  

We represent a morphism pair $(f,q)$ by the following commutative diagram:
\begin{equation} \label{CD:morphism}
\begin{CD}
	X_{1}     @>f>>    X_{2}  \\
      @Vp_{1}VV @VVp_{2}V \\
      T_{1} @>q>> T_{2}
\end{CD}
\end{equation}
The set of morphisms $(f,q)$ from $X_{1}$ into $X_{2}$ is denoted by
$C^{\infty,0}(X_{1},X_{2})$.  If in addition $T_{1}=T_{2}=T$ and $q=\text{id}$, then we usually 
just write $f$ in place of $(f,\text{id})$.  The composition of two morphisms $(f,q), (f',q')$ is the morphism $(f'\circ f,q'\circ q)$.  We note that 
$(f'\circ f)^{t}=(f')^{q(t)}\circ f^{t}$.

We write $C^{\infty,0}(X_{1})= C^{\infty,0}(X_{1},\C)$
where $\C$ is regarded as a continuous family over the one point space.  
Further, we set 
$C^{\infty,0}_{c}(X_{1},X_{2})=C^{\infty,0}(X_{1},X_{2})\cap C_{c}(X_{1},X_{2})$,
$C^{\infty,0}_{c}(X_{1})=C^{\infty,0}(X_{1})\cap C_{c}(X_{1})$.  
If $q$ is the identity map and $f$ is a homeomorphism such that both $f, f^{-1}$ are $C^{\infty,0}$ maps, then we say that $f$ is a {\em diffeomorphism}.

Suppose now that $T_{1}=T_{2}=T$, $q=\text{id}$ in (\ref{CD:morphism}), and let $(X,p)$ be a continuous family over $T$.  Then $*$-ing (\ref{CD:morphism}) with $X$ (noting that 
$X*T=X$) gives that $\text{id}*f$ is a morphism:
\begin{equation} \label{CD:morphism2}
\begin{CD}
	X*X_{1}     @>\text{id}*f>>    X*X_{2}  \\
      @Vt_{1}VV        @           VVt_{1}V \\
      X            @>\text{id}>>       X
\end{CD}
\end{equation}

Following the proof of the corresponding result for the 
$C^{\infty}$ case (\cite[1.2,1.3]{Helgason}), 
$C^{\infty,0}$-partitions of unity exist for continuous families.

We will require the notion of a $C^{\infty,0}$ complex
vector bundle over $X$, generalizing that given in the context of
\cite{AS4}.  (The real version is similar.)  Let $(E,\pi)$ be a (continuous) $p$-dimensional,  complex vector bundle 
over $X$ with fibers $E^{x}$.  Suppose further that $(E,p\circ \pi)$ is a continuous family of manifolds over $T$ with fibers $\mathcal{E}^{t}$.  
We say that $(E,\pi)$ is a {\em $C^{\infty,0}$ vector bundle} over $(X,p)$ if, for every $x\in X$, there exists an open neighborhood $U$ of $x$ in $X$ and a vector bundle local trivialization 
$h:\pi^{-1}(U)=E_{U}\to U\x \C^{p}$, where $h$ is $C^{\infty,0}$ over $p(U)$.  It is obvious that for each $t$,
$\mathcal{E}^{t}=E_{X^{t}}$ is a $C^{\infty}$-vector bundle over $X^{t}$.  As in the case 
of $C^{\infty}$-manifolds, a $C^{\infty,0}$ vector bundle is determined by an open cover 
$\{U_{\al}\}$ of $X$ and $C^{\infty,0}$-transition functions $g_{\al \bt}:U_{\al}\cap U_{\bt}\to
GL(p,\C)$ satisfying the usual compatibility relations.  Trivial $C^{\infty,0}$ vector bundles are of the form $X\x \C^{p}$.  It is left to the reader to check that pull-backs of $C^{\infty,0}$ vector bundles are themselves $C^{\infty,0}$ vector bundles.

A $C^{\infty,0}$-section of a $C^{\infty,0}$ vector bundle $E$ is a function $s'\in C^{\infty,0}(X,E)$ such that $s'(x)\in E^{x}$.  So $s'$ is a section of the continuous vector bundle $(E,\pi)$ 
which is also a morphism:
\begin{equation} \label{CD:morphs'}
\begin{CD}
	X     @>s'>>    E  \\
      @VpVV @VVp\circ \pi V \\
      T    @>\text{id}>> T
\end{CD}
\end{equation}
Global $C^{\infty,0}$-sections of $E$ are obtained in the usual way from local ones using a
$C^{\infty,0}$-partition of unity on $X$. It is easy to check that if $E_{1}, E_{2}$ are $C^{\infty,0}$ vector bundles over $X$, then $E_{1}\oplus E_{2}$ and $E_{1}\otimes E_{2}$ are $C^{\infty,0}$ vector bundles over $X$.  	

If $(X,p)$ is a continuous family over $T$, then, using transition functions as for the $C^{\infty}$-case, we obtain that $TX=\cup_{t\in T}TX^{t}$ is
a $C^{\infty,0}$ real vector bundle over $X$.  
A hermitian metric $\ga=\{\ga^{x}\}_{x\in X}$ on $E$ will be called {\em $C^{\infty,0}$} if for any
$C^{\infty,0}$ sections $s',t'$ of $E$, the function $x\to \ga(s'(x),t'(x))$
belongs to $C^{\infty,0}(X)$.  A $C^{\infty,0}$-hermitian metric $\ga$ on $E$ 
can alternatively be defined as a continuous hermitian metric which is a morphism in the following sense:
\begin{equation} \label{CD:hermitian}
\begin{CD}
	E\oplus E                    @>\ga>>    \C  \\
      @Vp\circ (\pi\oplus\pi) VV               @VVV\\
       T                            @>>>          *
\end{CD}
\end{equation}

$C^{\infty,0}$-hermitian metrics  are constructed in the same way that Riemannian
metrics are constructed for smooth manifolds, using local $C^{\infty,0}$-frames
and a $C^{\infty,0}$ partition of unity.  Locally, the vector bundle $E$ can be identified with some
$U\x \C^{p}$, and then a hermitian metric can be identified with a $C^{\infty,0}$-map
\begin{equation}     \label{eq:Auz}
x\to A(x)
\end{equation}
into $P_{p}$, the set of invertible, $p\x p$, positive definite matrices in $M_{p}(\C)$.  Precisely, let $(U,\phi)$ be a chart, $\phi(U)=p(U)\x W$ and $h:E_{U}\to U\x \C^{p}$ be a $C^{\infty,0}$
trivialization of $E_{U}$.  
Then 
\begin{equation}
(A(x)\xi,\eta)=\ga^{x}(h^{-1}(x,\xi),h^{-1}(x,\eta)).
\end{equation}
The dual $E^{*}=\cup_{x\in X}(E^{x})^{*}$ is identified with $E$ using 
a $C^{\infty,0}$-hermitian metric on $E$ in the usual way, so that $E^{*}$ is also a $C^{\infty,0}$ vector bundle over $X$.  In particular, $T^{*}X=(TX)^{*}$ is a $C^{\infty,0}$ vector bundle over $X$.


\section{Continuous family groupoids and their actions} 

A {\em groupoid} is most simply defined as a small category with inverses.  The unit space of $G$ is denoted by $G^{0}$, and the {\em range} and {\em source} maps $r:G\to G^{0}$, $s:G\to G^{0}$ are given by: $r(g)=gg^{-1}, s(g)=g^{-1}g$.    
The multiplication map $(g,h)\to gh$, denoted by $m$, is defined on the set $G^{2}$ of composable pairs 
$\{(g,h): s(g)=r(h)\}$.  The inversion map $g\to g^{-1}$ on $G$ will be denoted by $i$.

For detailed discussions of groupoids (including locally compact and Lie
groupoids), the reader is referred to the books
\cite{Mackenzie,MuhlyTCU,Paterson,rg}. Important examples of groupoids are given by
transformation group groupoids and equivalence relations.

A {\em locally compact groupoid} is a groupoid $G$ which is also a second countable locally compact Hausdorff space for which multiplication and
inversion are continuous. Note that $G^{2},G^{0}$ are closed subsets of $G\x G,
G$ respectively. Let $r^{-1}(\{u\})=G^{u}$ and
$s^{-1}(\{u\})= G_{u}$.   Since $r,s$ are continuous, both $G^{u}$ and
$G_{u}$ are closed subsets of $G$. 

A left Haar system on a locally compact groupoid $G$ is a family of measures $\{\la^{u}\}$ $(u\in G^{0})$ on $G$, where each $\la^{u}$ is a
positive regular Borel measure with support $G^{u}$, and for any $f\in C_{c}(G)$, the function $u\to f^{0}(u)=\int_{G^{u}}f\,d\la^{u}$ is continuous, and 
\[ \int_{G^{s(g)}}f(gh)\,d\la^{s(g)}(h) = \int_{G^{r(g)}}f(h)\,d\la^{r(g)}(h). \]
for all $g\in G$.  In the situation of the present paper, a left Haar system will always exist on $G$.   It follows that the maps $r,s:G\to G^{0}$ are open.

Set $\la_{u}=(\la^{u})^{-1}$ on $G_{u}$.  Then $C_{c}(G)$ is a convolution $^{*}$-algebra, where
\[  f*g(x)=\int f(t)g(t^{-1}x)\, d\la^{r(x)}(t),\hspace{.2in}  f^{*}(x)=\ov{f(x^{-1})}.       \]

We will require the reduced $\cs$ $C_{red}^{*}(G)$ of $G$.  There are two 
approaches to $C_{red}^{*}(G)$, both of which are useful.  For the first, for each $u\in
G^{0}$, the representation $\pi_{u}$ of $C_{c}(G)$
on the Hilbert space $L^{2}(G,\la_{u})$ is given by:
for $f\in C_{c}(G), \xi\in C_{c}(G)\subset L^{2}(G,\la_{u})$,
\begin{equation}
     \pi_{u}(f)(\xi)=f*\xi\in C_{c}(G)\subset L^{2}(G,\la_{u}).       \label{eq:red}
\end{equation}
The reduced $\cs$-norm on $C_{c}(G)$ can then (\cite[p.108]{Paterson})
be defined by:
\[        \norm{f}_{red}=\sup_{u\in G^{0}}\norm{\pi_{u}(f)},       \]
and $C_{red}^{*}(G)$ is the completion of $(C_{c}(G), \norm{.}_{red})$.

The second approach to $C_{red}^{*}(G)$ uses a Hilbert module 
(\cite[Proposition 8]{ag}).  Let $E^{2}$ be the completion of $C_{c}(G)$ under the norm:
$\norm{f}=\sup_{u\in G^{0}}\norm{f^{u}}_{2}$ where $f^{u}=f_{\mid G^{u}}$.
Let $D=C_{0}(G^{0})$.
Then $E^{2}$ is a Hilbert $D$-module with right action $(f,a)\to f\x (a\circ r)$ ($a\in D$) and inner product: $\lan f,f'\ran(u)=\lan f^{u},(f')^{u}\ran=\int \ov{f^{u}}(f')^{u}\,d\la^{u}$.  The map $f\to R_{f}$, where
$R_{f}(F)=F*f$, is a $^{*}$-antihomorphism from $C_{c}(G)$ into $\mfL(E^{2})$.  The closure of the algebra of operators $R_{f}$ in $\mfL(E^{2})$ is the reduced $\cs$ $C_{red}^{*}(G)$ of $G$. 

We now discuss locally compact groupoid actions.  Let $G$ be a locally compact groupoid and let the pair $(X,p)$ be such that $X$ is a locally compact Hausdorff space and $p$ is a continuous open map from $X$ onto $G^{0}$.  (The pair $(X,p)$ could be, but doesn't have to be, a continuous family of manifolds.)
Form the fibered product $G*X$ of $(G,s)$ and $(X,p)$: so $G*X=\{(g,x)\in G\x X: s(g)=p(x)\}$.
Then (cf.\cite{MuReW,MuW}) $X$ is called a {\em $G$-space} if there is given a continuous map $n:G*X\to X$, $(g,x)\to gx$, that satisfies the natural algebraic axioms: i.e. $p(gx)=r(g)$,
$g_{1}(g_{2}x)=(g_{1}g_{2})x$ and $g^{-1}(gx)=x$ whenever these make sense. We then say that $G$ {\em acts} on $X$.   

The action 
of $G$ on $X$ is called {\em proper} if the map $(g,x)\to (gx,x)$ is proper
from $G*X$ into $X\x X$ (inverse image of compact is compact). The basic results
for proper locally compact group actions (\cite{Palais,Phillips}) extend to the case of proper locally
compact groupoid actions. (See \cite{jsrc}.)  Indeed, suppose that the action of $G$ on $X$ is proper. Then the space $X/G$ of orbit equivalence classes is locally compact
Hausdorff in the quotient topology, and the quotient map $Q$ is open.  The space $X$ is called 
{\em $G$-compact} if it is of the form $GB$ for some compact subset $B$ of $X$.  This is equivalent to $X/G$ being compact.  

A locally compact groupoid $G$ is called a {\em Lie groupoid} (or a {\em smooth
groupoid}) if $G$ is a manifold such that $G^{0}$ is a submanifold of $G$, the
maps $r,s:G\to G^{0}$ are submersions, and the product and inversion maps for
$G$ are smooth. Note that if $G$ is a Lie groupoid, then $G^{2}$ is naturally a
submanifold of $G\x G$ and every $G^{u}, G_{u}$ is a submanifold of $G$. (See
\cite[pp.55-56]{Paterson}.)  Every Lie groupoid $G$ admits a {\em smooth}
left Haar system $\{\la^{u}\}$. This means that in addition to satisfying the
conditions for a left Haar system, each $\la^{u}$ is a strictly
positive smooth measure on the submanifold $G^{u}$ and the $\la^{u}$'s locally
vary smoothly. (For more details, see \cite[p.61]{Paterson}.)

Invariant
integration on a Lie groupoid can be formulated canonically in terms of a
$1/2$-density bundle.  This approach has been developed by H\"{o}rmander in the context of pdo's on manifolds (\cite[p.93f.]{Hor}) and for Lie groupoids by Connes 
(\cite[p.101]{Connesbook}).  We get a smooth left Haar
system by specifying a trivialization of the density bundle.
However, since the representation theory of locally compact groupoids is usually
developed in terms of left Haar systems, it is convenient to fix a smooth
left Haar system in advance.  (All smooth left Haar systems are equivalent in the natural way.)

Continuous family groupoids form a wider class of locally compact
groupoids than the Lie groupoids.  They are defined as follows (\cite{families}).

\begin{definition}           \label{def:sfg}
A locally compact groupoid $G$ is called a {\em continuous family groupoid} if:
\bi
\item[(i)] both $(G,s), (G,r)$ are continuous families of manifolds over
$G^{0}$;
\item[(ii)] the inversion map $i:(G,r)\to (G,s)$ is a diffeomorphism;
\item[(iii)] $G*G=G^{2}$ is the fibered product of the continuous families $(G,s)$ and $(G,r)$, and the pair $(m,r)$ is a morphism of continuous families from $(G*G,t_{1})\to (G,r)$:
\nopagebreak
\begin{equation}  \label{CD:GstarG}
\begin{CD}
	G*G  @>m>>           G  \\
      @Vt_{1}VV            @VVrV \\
	G       @>r>>        G^{0}
\end{CD}
\end{equation}
\ei
\end{definition}

In the preceding definition, we note that the map $m$ is a fiber preserving map from $(G*G,t_{1})$ into $(G,r)$ since $r(gh)=r(g)$.  We note also that from (ii), the $C^{\infty,0}$ structures on $(G,s), (G,r)$ mutually determine each other through the inversion map.  

Lie groupoids are obviously continuous family groupoids. Examples of continuous family groupoids that are not Lie groupoids are given in \S 6.

\begin{definition}
The pair $(X,p)$ (or simply $X$) is called a
{\em $C^{\infty,0}$ $G$-space} if $(X,p)$ is a $G$-space that is a continuous family
of manifolds and is such that the pair $(n,r)$ is a morphism from $(G*X,t_{1})$
into $(X,p)$:
\begin{equation}  \label{CD:GstarX}
\begin{CD}
	G*X     @>n>>     X  \\
      @Vt_{1}VV        @VVpV \\
	G       @>r>>     G^{0}
\end{CD}
\end{equation}
\end{definition}

What this definition says is that under the $G$-action, each $g\in G$ determines a diffeomorphism in $C^{\infty}(X^{s(g)},X^{r(g)})$ that varies continuously with $g$.  By (\ref{CD:GstarG}), $(G,r)$ is always a proper, $C^{\infty,0}$ $G$-space under left multiplication.

The morphism condition (\ref{CD:GstarX}) can be reformulated in terms of the pair $(n',s)$: 
\begin{equation}  \label{CD:GstarrX}
\begin{CD}
	G*_{r} X     @>n'>>     X  \\
      @Vt_{1}VV           @VVpV \\
	G       @>s>>     G^{0}
\end{CD}
\end{equation}
where $n'(g,x)=g^{-1}x$ and $G*_{r} X$ is the continuous family $(G,r)*(X,p)$.

\begin{proposition}    \label{prop:n'}
The pair $(n',s)$ is a morphism.
\end{proposition}
\begin{proof}
This follows from the morphism composition:
\begin{equation}  \label{CD:2GstarrX}
\begin{CD}
	G*_{r} X  @>\text{h}>> G*X   @>n>>     X  \\
      @Vt_{1}VV                 @VVt_{1}V       @VVpV \\            
	G           @>i>>          G     @>r>>     G^{0}
\end{CD}
\end{equation}
where $h=i*\text{id}$.
To prove this, because of (\ref{CD:GstarX}), we just have to show that $(h,i)$
is a morphism.  This is trivial since locally, $h$ is of the form $(g,w)\to (g^{-1},w)$.
\end{proof}

\begin{proposition}   \label{prop:grx}
Let $(X,p), (Y,q)$ be $C^{\infty,0}$ $G$-spaces and $f\in C^{\infty,0}(X,Y)$. Define
$f': G*_{r} X\to Y$ by:
\begin{equation}
f'(g,x)=g[f(g^{-1}x)].            \label{eq:gf}
\end{equation}
Then $(f',r)\in C^{\infty,0}(G*_{r} X, Y)$.
\end{proposition}
\begin{proof}
It is left to the reader to check that $(n'',\text{id})$ is a morphism, where 
$n''(g,x)=(g,g^{-1}x)$:
\begin{equation}  \label{CD:Gstarr2X}
\begin{CD}
	G*_{r} X     @>n''>>   G*X  \\
      @Vt_{1}VV           @VVt_{1}V \\
	G          @>\text{id}>>    G
\end{CD}
\end{equation}             
Using (\ref{CD:GstarX}) and (\ref{CD:morphism2}), we obtain that  
$f'$ is a morphism by morphism composition:
\begin{equation} \label{CD:compos}
\begin{CD}
  G*_{r} X     @>n''>>   G*X  @>\text{id}*f>>    G*Y     @>n>>   Y   \\
  @Vt_{1}VV           @VVt_{1} V      @VVt_{1}V       @VVqV  \\
  G       @>\text{id}>>     G    @>\text{id}>>     G     @>r>>   G^{0}  \\
\end{CD}
\end{equation}
\end{proof}

The notion of a {\em $C^{\infty,0}$ $G$-vector bundle} $(E,\pi)$ over a $C^{\infty,0}$
$G$-space $(X,p)$ is defined in the natural way.  
We require that $(E,\pi)$ be a $C^{\infty,0}$ vector bundle that is at the same time a continuous $G$-vector bundle over $X$, and is such that $(E,p\circ \pi)$ is a $C^{\infty,0}$ $G$-space.

The vector bundle $E^{*}$ is also a $C^{\infty,0}$ $G$-vector bundle using a $C^{\infty,0}$-invariant metric 
on $E$ (Proposition~\ref{prop:hermG} below) to identify $E$ with $E^{*}$ as $C^{\infty,0}$ $G$-vector bundles.  If $E,F$ are $C^{\infty,0}$ $G$-vector bundles over $X$, then $\text{Hom}(E,F)=F\otimes E^{*}$
is also a $C^{\infty,0}$ $G$-vector bundle over $X$ with the usual action: $gA(e)=gA(g^{-1}e)$.  
 
If $X$ is a $C^{\infty,0}$ $G$-space, then there is a natural action of $G$ on
$TX$ for which the $C^{\infty,0}$ vector bundle $TX$ (\S 3) is a $C^{\infty,0}$ $G$-vector bundle. Indeed, for $x\in X^{s(g)}, \xi\in TX^{s(g)}$, the action is given by $(g,\xi)\to 
\ell_{g}'(x)\xi$, where $\ell_{g}:X^{s(g)}\to X^{r(g)}$ is defined by: $\ell_{g}x=gx$.

As for Lie groupoids, there is a natural (effectively unique) kind of left Haar
system on a continuous family groupoid $G$. This is a $C^{\infty,0}$ left Haar
system (\cite{families}).  More generally, as we shall see (Proposition~\ref{prop:muu}), such a system exists for every proper $C^{\infty,0}$ $G$-space $(X,p)$.  For the rest of this section, $G$ is a continuous family groupoid, $(X,p)$ is a proper $C^{\infty,0}$ $G$-space, and $E$ is a $C^{\infty,0}$ $G$-vector bundle over $X$.

We now formulate the notion of a {\em $C^{\infty,0}$ left Haar system}
for $X$.  In this definition, $\phi^{u}=\phi_{\mid X^{u}\cap U}$ and 
$\mu^{u}\circ \phi^{u}(E)=\mu^{u}((\phi^{u})^{-1}E)$.

\begin{definition}  \label{def:slH}
A {\em $C^{\infty,0}$} left Haar system for $(X,p)$ is a family
$\{\mu^{u}\}_{u\in G^{0}}$ of smooth positive measures on the manifolds $X^{u}$ such that:
\bi
\item[(i)] the support of each $\mu^{u}$ is $X^{u}$;
\item[(ii)] for any chart $(U,\phi)$ for $(X,p)$, with $U\thicksim
p(U)\x W$, the measure
$\mu^{u}\circ \phi^{u}$ on $W$ is equivalent to the restriction $\la^{W}$ of Lebesgue measure to $W$, and the function
$(u,w)\to (d(\mu^{u}\circ \phi^{u})/d\la^{W})(u,w)$  belongs to
$C^{\infty,0}(p(U)\x W)$;
\item[(iii)] for any $g\in G$ and $f\in C_{c}(X)$,
\begin{equation} \label{eq:invarX}
\int_{X^{s(g)}}f(gx)\,d\mu^{s(g)}(x) = \int_{X^{r(g)}}f(x)\,d\mu^{r(g)}(x).
\end{equation}
\ei
\end{definition}

The weaker notion of a {\em continuous} left Haar system for $(X,p)$ is defined in the same way except that the 
$\mu^{u}$'s are only assumed to be positive, regular Borel measures, and
the local Radon-Nikodym derivatives in (ii) are only required to exist and be continuous.  
Note that any such system is a left Haar system in the earlier sense.  The following result is \cite[Theorem 1]{families}.  
\begin{theorem}  \hspace{.1in}        \label{th:LHS}
Let $G$ be a continuous family groupoid.  Then there exists a continuous left
Haar system $\{\la^{u}\}$ on $G$.
\end{theorem}
Now let $G$ be a continuous family groupoid and let $\{\la^{u}\}$ be a continuous left Haar system on $G$.  The function $c$ of the next proposition is the groupoid version of the ``cut-off'' function of Connes and Moscovici (\cite[p.295]{CoMos}).

\begin{proposition}    \label{prop:cgx}
Let $X$ be $G$-compact.  Then
there exists a non-negative $c\in C^{\infty,0}_{c}(X)$ such that
\begin{equation}
\int_{G^{p(x)}}c(g^{-1}x)\,d\la^{p(x)}(g)=1    \label{eq:cgm2}
\end{equation}
for all $x\in X$.
\end{proposition}
\begin{proof}
A $C_{c}(X)$-version of this result is given in \cite[Proposition 3.4]{jsrc} and this is simply modified as follows.  Let $X=GC$ where $C\in \mathcal{C}(X)$.  Using a $C_{c}^{\infty,0}$ partition of unity, there exists  $\xi\in C_{c}^{\infty,0}(X)$ such that $\xi\geq 0$, $\xi(x)>0$ for all $x\in C$.  Define $\eta:X\to \C$ by: 
$\eta(x)=\int_{G^{p(x)}}\xi(g^{-1}x)\,d\la^{p(x)}(g)$.  Then $\eta(x)>0$ for all $x\in X$.
We now show that $\eta\in C^{\infty,0}(X)$.

To this end, let $U\thicksim p(U)\x W$ be a chart in $X$, where $U$ is relatively compact.
Since the support of $\xi$ is compact and the $G$-action is proper, the set 
$\{g\in G: \xi(g^{-1}x)\ne 0, x\in U\}$ is contained in some $D\in \mathcal{C}(G)$.  Another partition of unity argument in $G$ gives that we can take 
$D\subset V\thicksim r(V)\x W'$, a chart in $G$.  Let $\psi\in C_{c}(V)$ be such that 
$0\leq \psi\leq 1$, $\psi=1$ on $D$.  From (\ref{CD:GstarrX}), the function
$g\to \xi(g^{-1}\cdot)$ is $C^{\infty}$-continuous.  Switching to coordinates, we can write
\[ \eta(u,w)=\int_{W} \psi(u,w')\xi(u,w',w)f(u,w')\,dw'                      \]
where $\psi, f$ are continuous and $(u,w')\to \xi(u,w',\cdot)$ is $C^{\infty}$-continuous.
Elementary analysis then gives that $\eta\in C^{\infty,0}(X)$.  Take $c=\xi/\eta$.
\end{proof}

We now discuss the $C^{\infty,0}$-version of a {\em $G$-partition of unity}.  Let 
$\{U_{\al}\}_{\al\in A}$ be a collection of open
subsets of $X$. A {\em $G$-partition of unity of $X$ subordinate to
$\{U_{\al}\}_{\al\in A}$} is a collection of
$C_{c}^{\infty,0}$ non-negative functions
$f_{\de}$ on $X$ ($\de\in \De$), each with compact support in some $U_{\al}$, such that for every $x\in X$, we
have
\begin{equation}  \label{eq:sumfga}
\sum_{\de}\int_{G^{p(x)}}f_{\de}(g^{-1}x)\,d\la^{p(x)}(g)=1
\end{equation}
where the sum is locally finite.  The next proposition asserts the existence of a $G$-partition
of unity.  The locally compact groupoid version of this is given in \cite[Proposition
3.3]{jsrc}. (Its proof parallels the locally compact group version of Phillips 
in \cite[Lemma 2.6]{Phillips}.)  The proof of the proposition follows that of the locally compact groupoid version, taking the functions
$u_{x,\bt}$ used in that proof to be $C^{\infty,0}$.

\begin{proposition}      \label{prop:aver}
Let $\{U_{\al}\}_{\al\in A}$ be a
collection of open subsets of $X$ such that the family of sets $GU_{\al}$ covers $X$. Then
there exists a $G$-partition of unity subordinate to the collection $\{U_{\al}\}_{\al\in A}$.
\end{proposition}

The following result is a $C_{c}^{\infty,0}$-version of \cite[Lemma 3.1]{jsrc}. It is proved as in Proposition~\ref{prop:cgx}.

\begin{proposition}        \label{prop:cont}
Let $f\in C^{\infty,0}(G*_{r}X)$.
Suppose that for all
$C\in \mathcal{C}(X)$ and with $W_{C}=\{(g,x)\in G*_{r} X: x\in C\}$, the
restriction $f_{\mid W_{C}}\in C_{c}(W_{C})$.
Then $F\in C^{\infty,0}(X)$, where
\begin{equation}      \label{eq:f'F}
F(x)=\int f(g,x)\,d\la^{p(x)}(g).
\end{equation}
\end{proposition}

The Lie groupoid version of the following proposition (with additional assumptions on $G$ and $(X,p)$) is given in
\cite[Proposition 4.2]{jsrc}. (The locally compact group
version is given in \cite[pp.40-41]{Phillips}.)  We require first a lemma.

\begin{lemma} \label{lemma:lan'}
Let $\ga'$ be a $C^{\infty,0}$-hermitian metric on a $C^{\infty,0}$ $G$-vector bundle $E$ (\S 2) and
$s_{1}, s_{2}:X\to E$ be $C^{\infty,0}$-sections of $E$. Let 
$F:G*_{r} X\to \C$ be given by:
\[      F(g,x)=\ga'(g^{-1}s_{1}(x),g^{-1}s_{2}(x)).         \]
Then $F$ is a morphism on $(G*_{r} X,t_{1})$.
\end{lemma}
\begin{proof} Define, for $i=1,2$, the map $\bt_{i}:G*_{r} X\to E$ by:
$\bt_{i}(g,x)=g^{-1}s_{i}(x)$.  Since 
$\bt_{i}=n'\circ (\text{id }* s_{i})$, it follows using  
(\ref{CD:morphs'}) and (\ref{CD:morphism2}) that each $(\bt_{i},s)$ is a morphism:
\[
\begin{CD}
     G*_{r} X  @>>\text{id}*s_{i}>  G*_{r} E   @>>n'> E \\
     @Vt_{1} VV          @Vt_{1}VV             @VV p\circ \pi V \\
     G            @>\text{id}>>        G      @>s>>       G^{0}
\end{CD}
\]
Let $\al:G*_{r} X\to E\oplus E$ be given by: $\al(g,x)=(\bt_{1}(g,x),\bt_{2}(g,x))$. 
Then $F$ is a morphism being the composition of two morphisms 
(see (\ref{CD:hermitian})):
\[
\begin{CD}
     G*_{r} X  @>\al>>   E\oplus E   @>\ga'>>       \C \\
     @Vt_{1} VV               @VVV                @VVV \\
     G            @>s>>        G^{0}      @>>>       *
\end{CD}
\]
\end{proof}

\begin{proposition}   \label{prop:hermG}
There exists a $C^{\infty,0}$-hermitian metric $\ga$ on $E$
which is $G$-isometric.
\end{proposition}
\begin{proof}
Let $\ga'$ be a $C^{\infty,0}$-hermitian metric on $E$ and
$\{f_{\de}\}$ be a $G$-partition of unity for $X$ (so $\{U_{\al}\}_{\al\in A}=\{X\}$).  Let $s_{1}, s_{2}\in C^{\infty,0}(X,E)$.
Given $\de\in \De, x\in X$, the function
\[  f(g,x) =
f_{\de}(g^{-1}x)\ga'(g^{-1}s_{1}(x),g^{-1}s_{2}(x))  \]
satisfies the conditions of Proposition~\ref{prop:cont} using Proposition~\ref{prop:grx}
and Lemma~\ref{lemma:lan'}.  So the map
$x\to \int_{G^{p(x)}}f(g,x)\,d\la^{p(x)}(g)$ is 
$C^{\infty,0}$.   The local finiteness of the $\{f_{\de}\}$ then yields that $\ga$ is a 
$C^{\infty,0}$-hermitian metric where
\begin{equation}
  \ga(\xi,\eta)=
\sum_{\de}\int_{G^{p(x)}}f_{\de}(g^{-1}x)\ga'(g^{-1}\xi,g^{-1}\eta)\,d\la^{p(x)}(g).    \label{eq:lan'}
\end{equation}
To prove that $\ga$ is $G$-isometric, one argues:
\[\ga(h\xi,h\eta)=
\sum_{\de}\int f_{\de}(g^{-1}hx)\ga'(g^{-1}h\xi,g^{-1}h\eta)\,d\la^{r(h)}(g)=\ga(\xi,\eta).  \]
\end{proof}

\begin{proposition}   \label{prop:muu}
There exists a $C^{\infty,0}$ left Haar system $\{\mu^{u}\}$ for $(X,p)$.
\end{proposition}
\begin{proof}
We apply Proposition~\ref{prop:hermG} with $E=TX$ to obtain that there exists
a (real) $C^{\infty,0}$-hermitian metric $\ga$ on $TX$ which is $G$-isometric.
In the standard way, the metric determines a family $\mu^{u}$ of smooth
measures on the $X^{u}$'s.  In terms of local coordinates ((\ref{eq:Auz})),
$d\,\mu^{u}(x)=\left|\det A(u,w)\right|^{1/2}\negthinspace dw$ ($x\sim (u,w)$). The $G$-isometric property of $\ga$ 
gives (as in \cite[p.122]{jsrc}) that
the differential $\ell_{g}':TX^{s(g)}\to TX^{r(g)}$ is an isometry at every
$x\in X^{s(g)}$.  This gives the $G$-invariance of the $\mu^{u}$'s. The remaining
axioms for a $C^{\infty,0}$ left Haar system are easily shown to be satisfied
by the $\mu^{u}$'s.
\end{proof}

\begin{corollary}   \label{cor:slhs}
The continuous family groupoid $G$ has a $C^{\infty,0}$ left Haar system.
\end{corollary}


\section{Continuous families of pseudodifferential operators}

In this section we discuss briefly results on continuous families of
pseudodifferential operators (pdo's). It seems likely that the theory
sketched below is known, but for lack of a suitable reference, some
discussion of it seems appropriate.  The pseudodifferential operators are of
the kind $L^{m}_{\rho,\de}$ that were introduced and investigated in detail by
L. H\"{o}rmander (\cite{Hor2}). (In particular, the {\em classical} or
{\em polyhomogeneous} pdo's are included as a special case.   Smooth families of such pdo's 
on smooth groupoids have been investigated by Lauter, Monthubert and Nistor (\cite{LMN00,LMN01,LN}).)  We shall
primarily rely on the exposition of the theory of such pdo's operators
given by M. A. Shubin (\cite{Shubin}).  I am grateful to Professor Shubin for helpful correspondence.

Throughout the discussion, $m,\rho,\de$ will be real numbers such that
\begin{equation}
0\leq 1-\rho\leq \de < \rho. \label{eq:rhode}
\end{equation}
Let $m, N\geq 1$ and $W$ be an open subset of $\R^{m}$. Then $S_{\rho,\de}^{m}(W\x \R^{N})$
is the set of functions $a\in C^{\infty}(W\x \R^{N})$ such that for all
$K\in \mathcal{C}(W)$ and  multi-indices
$\al ,\bt$, there exists a constant $C_{\al ,\bt,K}\geq 0$ such that
for all $x\in K$ and $\theta\in \R^{N}$, we have
\begin{equation}
\left|\partial^{\al}_{\theta}\partial^{\bt}_{x}a(x,\theta)\right|
\leq C_{\al ,\bt,K}\left<\theta\right>^{m-\rho\left|\al\right|+
\de\left|\bt\right|}
\label{eq:dalbt}
\end{equation}
where $\left<\theta\right>=(1+\left|\theta\right|^{2})^{1/2}$.  The vector space
$S_{\rho,\de}^{m}=S_{\rho,\de}^{m}(W\x \R^{N})$
is a Fr\'{e}chet space under the seminorms
$\left\|a\right\|_{\al ,\bt,K}$, where $\left\|a\right\|_{\al ,\bt,K}$
is the best constant $C_{\al ,\bt,K}$ satisfying
(\ref{eq:dalbt}).  The space $S_{\rho,\de}^{-\infty}$
is defined to be $\cap_{m}S_{\rho,\de}^{m}$.  (It does not depend on $\rho$ and $\de$.)  If
$a(x,\theta)\in S_{\rho,\de}^{m}(W\x \R^{N})$, then
$\partial^{\al}_{\theta}\partial^{\bt}_{x}a(x,\theta)\in
S_{\rho,\de}^{m-\rho\left|\al\right| + \de\left|\bt\right|}$.
Further, if $a\in S_{\rho,\de}^{m}, b\in S_{\rho,\de}^{m'}$, then
$ab\in S_{\rho,\de}^{m+m'}$.

Now let $a(x,y,\xi)\in S_{\rho,\de}^{m}(W\x W\x \R^{N})$ where $W$ is an open subset of $\R^{N}$.  Then the function
$a$ determines a pdo $A:C_{c}^{\infty}(W)\to C^{\infty}(W)$ given by:
\[    Af(x)=(2\pi)^{-N}\iint e^{\imath (x-y)\cdot\xi}a(x,y,\xi)f(y)\,dy\,d\xi.\]
We write $A\in L_{\rho,\de}^{m}(W)$.  Each such $A$ determines, by the
Schwartz kernel theorem, a distribution $K_{A}(x,y)\in \mathcal{D}'(W\x W)$ where 
(\cite[p.11]{Shubin})
\begin{equation}  \label{eq:kaw}
\left<K_{A},w\right>=(2\pi)^{-N}\iiint 
e^{\imath (x-y)\cdot \theta}a(x,y,\theta)w(x,y)\,dx\,dy\,d\theta.
\end{equation}
Further, for $f,g\in
C_{c}^{\infty}(W)$, we have
\[    \left<Af,g\right>=\int (Af)\ov{g}=\left<K_{A}, f(y)\ov{g(x)}\right>.            \]
Then $A$ is determined by $K_{A}$ and conversely.
(Of course, $A$ does not determine $a(x,y,\xi)$ uniquely.)   The support of $K_{A}$ is denoted by $supp\: A$ ($\subset W\x W$).  Let $T$ be as in \S 2.

\begin{definition}
A {\em continuous family} (over $T$) with values in $S_{\rho,\de}^{m}(W\x\R^{N})$ is a continuous map
$t\to a(t,x,\theta)$ from $T$ into $S_{\rho,\de}^{m}(W\x\R^{N})$.
\end{definition}

The space of such continuous families is denoted by
$S_{\rho,\de}^{m}(T;W\x\R^{N})$, and is a Fr\'{e}chet space under the
seminorms $a\to \left\|a\right\|_{C;\al ,\bt,K}$, where for $C\in
\mathcal{C}(T)$, we set
\[          \left\|a\right\|_{C;\al ,\bt,K}=
\sup_{t\in C}\left\|a^{t}\right\|_{\al ,\bt,K}    \]
where $a^{t}(x,y,\theta)=a(t,x,y,\theta)$.
Of course, we take
$S_{\rho,\de}^{-\infty}(T;W\x \R^{N})
=\cap_{m}S_{\rho,\de}^{m}(T;W\x \R^{N})$.  Clearly, from (\ref{eq:dalbt}),
$S_{\rho,\de}^{m}(T;W\x \R^{N})\subset C^{\infty,0}(T\x W\x \R^{N})$.
We now extend in the obvious way the notion of an {\em asymptotic expansion}
for a function in $S_{\rho,\de}^{m}(T;W\x\R^{N})$.

Let $F_{j}\in S_{\rho,\de}^{m_{j}}(T;W\x\R^{N})$ for $j=1,2,\ldots $ where
$m_{j}\to -\infty$. Let $F\in C^{\infty,0}(T\x W\x \R^{N})$.  We write
\[  F(t,x,\theta)\thicksim \sum_{j=1}^{\infty} F_{j}(t,x,\theta)         \]
if, for all $r\geq 2$,
\begin{equation}
F-\sum_{j=1}^{r-1} F_{j}\in S_{\rho,\de}^{\ov{m_{r}}}(T;W\x\R^{N}) \label{eq:asym}
\end{equation}
where $\ov{m_{r}}=\max_{j\geq r}m_{j}$.

The continuous families version of \cite[Proposition 3.5]{Shubin} holds, i.e. 
{\em given any sequence $\{F_{j}\}$, $F_{j}\in S_{\rho,\de}^{m_{j}}(T;W\x\R^{N})$ where
$m_{j}\to -\infty$, then there exists a function 
$F\in S_{\rho,\de}^{k}(T;W\x \R^{N})$, where $k=\max_{j}m_{j}$, such that 
$F\thicksim \sum_{j=1}^{\infty}F_{j}$, and $F$ is unique up to an element of  
$S_{\rho,\de}^{-\infty}(T;W\x\R^{N})$.}  The proof of \cite[Proposition 3.5]{Shubin} is the case where $T$ is a singleton.  The estimates in this proof also apply to the case where $T$ is compact, and the general case reduces to this using the $\si$-compactness of $T$.   

We now discuss continuous families of pseudodifferential operators.
We say that a family $t\to A^{t}$, abbreviated simply to $A$,
is a {\em continuous family in
$L_{\rho,\de}^{m}(W)$} if there exists
$a\in S_{\rho,\de}^{m}(T;W\x W\x\R^{N})$ such that for all
$f\in C_{c}^{\infty}(W)$, we have
\begin{equation}
A^{t}f(x)=
(2\pi)^{-N}\iint e^{\imath(x-y).\xi}a(t,x,y,\xi)f(y)\,dy\,d\xi.
\label{eq:pdo}
\end{equation}
The set of such continuous families is denoted by $L_{\rho,\de}^{m}(T\x W)$.  We write
$L_{\rho,\de}^{m}(T\x W;T\x W\x \C^{p}, T\x W\x \C^{q})$ for the space of $q\x p$ matrices whose entries are in $L_{\rho,\de}^{m}(T\x W)$. 

For $f\in C_{c}^{\infty,0}(T\x W)$, we write $Af(t,w)=A^{t}f^{t}(w)$.
By using a standard regularization (\cite[p.5]{Shubin}) if $f\in C^{\infty,0}_{c}(T\x W)$,
then the function $Af$ can be represented as an absolutely convergent integral whose integrand varies continuously in the parameter $t$ and smoothly in the parameter $x$.  
It follows that $Af\in C^{\infty,0}(T\x W)$.

The {\em kernel} of the family $A$ is the set of distributions
$\{K^{t}:t\in T\}$ where $K^{t}$ is the kernel of $A^{t}$.  The {\em support}
$supp\: A$ of $A$ is defined:
\[  supp\: A=\ov{\cup_{t\in T}\{t\}\x supp\: K^{t}}\subset T\x W\x W.   \]

The continuous family $A$ of pdo's is said to be {\em properly supported}
if the projection maps $\Pi_{1},
\Pi_{2}:supp\: A\to T\x W$ are proper, where
$\Pi_{1}(t,x,y)=(t,x), \Pi_{2}(t,x,y)=(t,y)$.  (The
non-local version of this for an almost differentiable groupoid is given in \cite{NWX}.  The non-local version in our context will be given later.) Obviously, if $A$ is properly supported, then so
is every $A^{t}$ in the usual sense of the term (\cite[p.16]{Shubin}).  If $A$ is properly supported,  then $A:C_{c}^{\infty,0}(T\x W)\to C_{c}^{\infty,0}(T\x W))$. 

Let $a\in S_{\rho,\de}^{m}(T;W\x W\x\R^{N})$ and $p:T\x W\x W\x \R^{N}\to T\x W\x W$
be the projection map: $(t,x,y,\xi)\to (t,x,y)$.  Define $supp_{t,x,y}\:a$ to be the closure of  
$p(supp\:a)$ in $T\x W\x W$.   We say that $a(t,x,y,\xi)$ is {\em properly supported}
(cf. \cite[p.18]{Shubin}) if both projections $(t,x,y)\to (t,x)$, $(t,x,y)\to (t,y)$ are proper on 
$supp_{t,x,y}\:a$.  As in \cite[Proposition 3.2]{Shubin}, if $A$ is properly supported, then the $a$ of (\ref{eq:pdo}) can be taken to be properly supported.  

Let $A\in L_{\rho,\de}^{m}(T\x W)$ be properly supported.  Let 
$e_{\xi}(x)=e^{\imath x.\xi}$.  Then the function $\si_{A}$, where
\[       \si_{A}(t,x,\xi)=e_{-\xi}(x)A^{t}e_{\xi}(x)     \]
belongs to $C^{\infty,0}(T\x W)$.  The function $\si_{A}$ is called the {\em symbol} of $A$.
Using the inverse Fourier transform
for $f^{t}$ in (\ref{eq:pdo}), the function $Af(t,x)$ is 
(cf. \cite[p.19]{Shubin}) the iterated integral:
\begin{equation}
Af(t,x)=(2\pi)^{-N}\iint e^{\imath(x-y).\xi}\si_{A}(t,x,\xi)f(t,y)\,dy\,d\xi
    \label{eq:asi}
\end{equation}

In fact, the argument of \cite[pp.21-25]{Shubin} adapts to give that 
$\si_{A}\in S_{\rho,\de}^{m}(T;W\x \R^{N})$ and if
$a\in S_{\rho,\de}^{m}(T;W\x W\x\R^{N})$ satisfies (\ref{eq:pdo}), then
\begin{equation}
\si_{A}(t,x,\xi)\thicksim \sum_{\al}\frac{1}{\al!}\partial^{\al}_{\xi}D^{\al}_{y}
a(t,x,y,\xi)|_{y=x}   \label{eq:siasy}
\end{equation}
where $D=\imath^{-1}\frac{\partial}{\partial y}$.  Note
also that
$\partial^{\al}_{\xi}D^{\al}_{y} a(t,x,y,\xi)_{\mid y=x}\in
S_{\rho,\de}^{m-(\rho -\de)\left|\al\right|}(T;W\x \R^{N})$ so that the
asymptotic expansion in (\ref{eq:siasy}) makes sense.

Next (cf. \cite[Proposition 3.3]{Shubin}), for {\em any}
$A\in L_{\rho,\de}^{m}(T\x W)$, we can write $A=A_{0} + A_{1}$ where
$A_{0}, A_{1}\in L_{\rho,\de}^{m}(T\x W)$ with $A_{0}$ properly supported and
$A_{1}$ with kernel in $C^{\infty,0}(T\x W\x W)$.  It follows that for such an $A$, we
can still define the symbol $\si_{A}$ as an equivalence class of
$S_{\rho,\de}^{m}(T;W\x \R^{N})/S_{\rho,\de}^{-\infty}(T;W\x \R^{N})$.

As in the case of a single pdo (\cite[p.26-28]{Shubin}),
if $A\in L_{\rho,\de}^{m}(T\x W)$ is properly
supported, then the transpose $^{t}A$ and the adjoint $A^{*}$ both belong to
$L_{\rho,\de}^{m}(T\x W)$ and are properly supported.  Further,
if $A\in L_{\rho,\de}^{m_{1}}(T\x W), B\in L_{\rho,\de}^{m_{2}}(T\x W)$ and $B$
is properly supported, then both $AB, BA$ belong to
$L_{\rho,\de}^{m_{1}+m_{2}}(T\x W)$.  The asymptotic expansions for $^{t}A,
A^{*}$ and $BA$ (both $B,A$ properly supported) correspond in the obvious way
to the asymptotic expansions for the corresponding single pdo cases.

Turning to change of variables for pdo's, let
$W,W_{1}$ be open subsets of $\R^{N}$, and $X,X_{1}$ be the trivial continuous families 
$X=T\x W$, $X_{1}=T\x W_{1}$.  Let $\ka:X\to X_{1}$ be a $C^{\infty,0}$-diffeomorphism from $X$ onto $X_{1}$, and let $\ka_{1}=\ka^{-1}$.  Abusing notation slightly, we write 
$\ka_{1}(t,w)=\ka_{1}^{t}(w)$.  
Let $A\in L_{\rho,\de}^{m}(X)$.  Then $A$
in ``$X_{1}$'' terms is given by the map
$A_{1}:C_{c}^{\infty,0}(X_{1})\to C^{\infty,0}(X_{1})$, where
\begin{equation}    \label{eq:A1kappa}
A_{1}(f)=A(f\circ \kappa)\circ \kappa_{1}.
\end{equation}

Further, $A_{1}$ is given by the family of Fourier integral operators:
\begin{equation}         \label{eq:a11}
A_{1}f(t,w)=(2\pi)^{-N}\iint e^{\imath(\ka_{1}(t,w)-\ka_{1}(t,z)).\xi}
a(t,\ka_{1}(t,w),\ka_{1}(t,z),\xi)\left|\det\,\ka_{1}'(t,z)\right|
f(t,z)\,dz\,d\xi       
\end{equation}
where $\ka_{1}'(t,z)$ is the Jacobian matrix of $\ka_{1}$ with respect to $z$ ($t$ fixed).
The single operator argument of \cite[pp.32-35]{Shubin} adapts readily to give 
$A_{1}\in L_{\rho,\de}^{m}(X_{1})$, and also its asymptotic expansion.  In particular, the leading term of that expansion is 
$\si_{A}(t,\ka_{1}(t,y),(^{t}\ka_{1}'(t,y))^{-1}\eta)\in 
S_{\rho,\de}^{m}(T;W_{1}\x \R^{N})$.
Further, as in the case of a 
single pdo (\cite[p.35]{Shubin}), it follows that 
\begin{equation}
\si_{A_{1}}(t,y,\eta)-\si_{A}(t,\ka_{1}(t,y),(^{t}\ka_{1}'(t,y))^{-1}\eta)
\in S_{\rho,\de}^{m-2(\rho -1/2)}(T;W_{1}\x \R^{N}). \label{eq:prin}
\end{equation}
So modulo symbols of order lower than $m-2(\rho -1/2)$, the symbols of all operators $A_{1}$ give the same well-defined function on the cotangent bundle $T^{*}(T\x W)$.

We will consider general continuous families of pdo's later in this section. For the
present we recall briefly some facts about a single pdo $A$ on a smooth manifold $M$.  Let $A:C_{c}^{\infty}(M)\to C^{\infty}(M)$ be a linear
map.  Suppose that for each chart $(U,\phi)$ for $M$, the map
$A_{1}:C_{c}^{\infty}(\phi(U))\to C^{\infty}(\phi(U))$ belongs to
$L_{\rho,\de}^{m}(\phi(U))$, where $A_{1}$ is given by (\ref{eq:A1kappa}) (with $\phi$ in place of $\ka$.)
Then $A$ is called a (scalar) pdo on $M$.  The set of such $A$'s is denoted by
$L_{\rho,\de}^{m}(M)$.

The notion of an $S_{\rho,\de}^{m}$-function extends to any smooth vector
bundle $E$ of rank $p$ over $M$.  Such a function is a section of $E$ which locally is given by a $p$-tuple of 
$S_{\rho,\de}^{m}$ functions.  The space of such functions is denoted by 
$S_{\rho,\de}^{m}(E)$.  

More generally, of course, one considers two smooth vector bundles $E,F$
over $M$ and a pdo $A:C_{c}^{\infty}(M,E)\to C^{\infty}(M,F)$. This means that locally, 
$A$ is given by a matrix of scalar pdo's.  The set of such $A$'s
is denoted by $L_{\rho,\de}^{m}(M;E,F)$.  Let $A\in L_{\rho,\de}^{m}(M;E,F)$.
Using a partition of unity argument to piece together the leading terms in the asymptotic expansions of local versions of $A$ ((\ref{eq:prin}))
gives an element of $S_{\rho,\de}^{m}(\text{Hom}(\pi^{*}E,\pi^{*}F))$, which is unique
modulo $S_{\rho,\de}^{m-2(\rho-1/2)}(\text{Hom}(\pi^{*}E,\pi^{*}F))$,
where $\pi:T^{*}M\to M$ is the canonical map.  The equivalence class of this element is called the {\em principal symbol} of $A$.  All of this extends naturally to the continuous families case below.

Indeed, let $(X,p)$ be a continuous family of manifolds over some $T$ and
$E,F$ be $C^{\infty,0}$ complex vector bundles of rank $p,q$ over $X$.  For each
$u\in T$, let $D^{u}\in L_{\rho,\de}^{m}(X^{u};E^{u},F^{u})$ and assume that $u\to D^{u}$ is continuous.  Let $D=\{D^{u}\}$.  For 
$U$ open in $X$, let $D_{U}=D_{\mid C_{c}^{\infty,0}(U,E_{U})}$.  
Given a chart $(U,\phi)$, $\phi(U)=p(U)\x W$, for $X$ trivializing both $E, F$, the family $D$ induces a map $(D_{U})_{1}\in L_{\rho,\de}^{m}(Y;Y\x \C^{p}, Y\x \C^{q})$ where
$Y=p(U)\x W$.
Let us precisely specify the family $(D_{U})_{1}$.  Let $h:E_{U}\to U\x \C^{p}$ and 
$h':F_{U}\to U\x \C^{q}$ be $C^{\infty,0}$-trivializations for $E_{U}, F_{U}$.  Let 
$\al:(\phi\x 1)\circ h:E_{U}\to \phi(U)\x \C^{p}$ and 
$\bt=(\phi\x 1)\circ h':F_{U}\to \phi(U)\x \C^{q}$.  Then 
\begin{equation}      \label{eq:du1}
(D_{U})_{1}f=\bt\circ D_{U}(\al^{-1}\circ f\circ\phi)\circ \phi^{-1}.
\end{equation} 

By 
Proposition~\ref{prop:muu} with $G=T$, the groupoid of units acting in the obvious way on $X$, there exists a $C^{\infty,0}$ left Haar system $\{\mu^{u}\}$ for $X$.  We can then define (as, e.g., in
\cite[p.21]{Eg}) the kernel $K^{u}$ of each $D^{u}$ by:
\[   D^{u}f(x_{1})=\int_{X^{u}} K^{u}(x_{1},x_{2})f(x_{2})\,d\mu^{u}(x_{2}).         \]
The kernel $K$ of $D$ is defined to be the set
$\{K^{u}: u\in T\}$.
The support $K$ of $D$ and the {\em properness} of $D$ are then defined as above in the case of $L_{\rho,\de}^{m}(T\x W)$.  In particular, 
$\text{supp }D=\ov{\cup_{u\in T}\text{supp }K^{u}}$ is a closed subset of $X*X$, and for properness, we require that the projections 
$(x_{1},x_{2})\to x_{1}$, $(x_{1},x_{2})\to x_{2}$ from $\text{supp }K$ into $X$ be proper maps.
The family D is said to have {\em compact support} or to be {\em compactly supported} if 
$\text{supp }D$ is compact.

\begin{definition}      \label{def:psdo}
The family $D$ is called a {\em pseudodifferential family} on $X$ if it satisfies the following local condition: given a chart $U\sim p(U)\x W=Y$ of $X$
trivializing $E$ and $F$ then the family
$(D_{U})_{1}$ belongs to $L_{\rho,\de}^{m}(Y;Y\x \C^{p},Y\x \C^{q})$.
\end{definition}

The set of pseudodifferential families $D$ is denoted by $L_{\rho,\de}^{m}(X;E,F)$. 

Note that in Definition~\ref{def:psdo}, we only require the condition on $(D_{U})_{1}$ to be valid for a family of charts $U$ forming a basis for the topology of $X$.  Note also that if $U$ is a chart and the support of the kernel of $D$ is contained in $U*U$, then $D$ and $(D_{U})_{1}$ mutually determine each other.  Familiar properties of $C^{\infty}$-pdo's extend to pseudodifferential families using similar proofs. In particular, $D$ is a pseudodifferential family if and only if $\phi D\psi$ is a pseudodifferential family for all $\phi,\psi\in C_{c}^{\infty,0}(X)$.  Further, if $D$ is a pseudodifferential family and is proper, then  $D:C_{c}^{\infty,0}(X,E)\to C_{c}^{\infty,0}(X,F)$. 
The principal symbol $\si_{D}$ of $D$ is defined as for a single pdo $A$ above; it belongs
to $S_{\rho,\de}^{m}(\text{Hom}(\pi^{*}E,\pi^{*}F))/S_{\rho,\de}^{m-2(\rho-1/2)}(\text{Hom}(\pi^{*}E,\pi^{*}F))$.  

Now let $X$ be a proper $C^{\infty,0}$ $G$-space for some continuous family groupoid $G$, and $E,F$ be $C^{\infty,0}$ $G$-vector bundles over $X$.  We can form the pull-back continuous family 
$(r^{*}X,t_{1})$ over $G$.  In the obvious way, this pull-back continuous family is just 
$(G *_{r} X,t_{1})$.  Also, each $(r^{*}X)^{g}=\{g\}\x X^{r(g)}$ which we will identify with $X^{r(g)}$.
Further, the pull-back bundle $r^{*}E$ is a $C^{\infty,0}$ vector bundle over $r^{*}X$.   

Each $f\in C^{\infty,0}(X,E)$ determines, for each $u\in G^{0}$, a $C^{\infty}$-section $f^{u}$ of $\mathcal{E}^{u}$.  For $g\in G$, we write
\begin{equation}  \label{eq:Lg}
L_{g}f(x)=L_{g}f^{s(g)}(x)=g[f(g^{-1}x)]\hspace{.1in} (x\in X^{r(g)}).
\end{equation}
The section $f$ is called {\em invariant} if 
$L_{g}f^{s(g)}=f^{r(g)}$ for all $g\in G$.  
  
Now let $P\in L_{\rho,\de}^{m}(X;E,F)$ be proper.  For each $g\in G$, define the pdo 
$\widetilde{P}^{g}$ by setting:
\[ \widetilde{P}^{g}=L_{g}P^{s(g)}L_{g^{-1}}\in
L_{\rho,\de}^{m}(X^{r(g)};\mathcal{E}^{r(g)},\mathcal{F}^{r(g)})=
L_{\rho,\de}^{m}((r^{*}X)^{g};(r^{*}E)^{g},(r^{*}F)^{g}). \]  

We now calculate how the local symbols and distributional kernel $K$ of $P$ relate to those of the $\widetilde{P}^{g}$'s.  Let $\{\mu^{u}\}$ be a $C^{\infty,0}$ left Haar system for the (proper $G$-space) $X$ (Proposition~\ref{prop:muu}).
For $h\in C_{c}^{\infty}(X^{r(g)},\mathcal{E}^{r(g)})$, 
\begin{gather*}
\widetilde{P}^{g}(h)(x)=L_{g}P^{s(g)}L_{g^{-1}}h(x)
=g(P^{s(g)}L_{g^{-1}}h)(g^{-1}x)\\
=\int gK^{s(g)}(g^{-1}x,y)L_{g^{-1}}h(y)\,d\mu^{s(g)}(y)=
\int gK^{s(g)}(g^{-1}x,g^{-1}z)g^{-1}h(z)\,d\mu^{r(g)}(z).  
\end{gather*}
So if $K^{g}$ is the distributional kernel of $\widetilde{P}^{g}$ then 
\begin{equation}     \label{eq:kgdist}
K^{g}(x,y)=gK^{s(g)}(g^{-1}x,g^{-1}y)g^{-1}. 
\end{equation}

Next, a trivial modification of the proof of (\ref{eq:prin}) (with $\kappa=\ell_{g}$) gives that
\begin{equation}      \label{eq:*}
\si_{L_{g}P^{s(g)}L_{g^{-1}}}(x,\eta)=g\si_{P^{s(g)}}(g^{-1}x,g^{-1}\eta)g^{-1}
\end{equation}
modulo
$S_{\rho,\de}^{m-2(\rho -1/2)}$, where $g^{-1}\eta=(^{t}\ell_{g^{-1}}')^{-1}\eta$.

Now let $\widetilde{P}=\{\widetilde{P}^{g}\}$.

\begin{proposition}  \label{prop:wdd}
$\widetilde{P}\in L_{\rho,\de}^{m}((r^{*}X);(r^{*}E),(r^{*}F))$ and is proper.
\end{proposition}
\begin{proof}
We can reduce the proof to the local case as follows.  Let
$(g_{0},x_{0})\in G*_{r} X$. We can find charts $Z, U, V$ around  $g_{0}$ in
$G$, $x_{0}$ in $X$ and $g_{0}^{-1}x_{0}$ in $X$ such that $Z^{-1}U\subset V$.  We
can further assume that $r(Z)=p(U)$ so that $Z*_{r} U$ is a chart containing
$(g_{0},x_{0})$ in $r^{*}X$.  We write $U\thicksim p(U)\x W$, $Z\thicksim
r(Z)\x L$ and $V\thicksim p(V)\x W'$. In addition, we can assume that $E,F$
are trivial over both $U,V$ so that both $r^{*}E, r^{*}F$ are trivial over
$Z*_{r} U$.   Also $Z*_{r} U\thicksim Z\x W$. We will use local coordinates. 

Let $P^{u}$ be represented as in (\ref{eq:pdo}).
Then as in (\ref{eq:a11}) and with $J_{g^{-1}}$ the Jacobian of the map $z\to g^{-1}z=y$, we get
for $x\in U, h\in C_{c}^{\infty,0}(U,E_{U})$ and $g\in Z$, 
\beqns
L_{g}P^{s(g)}L_{g^{-1}}h(x) & = &
g[P^{s(g)}(L_{g^{-1}}h)(g^{-1}x)]  \\
& = &
(2\pi)^{-k}g\iint e^{\imath(g^{-1}x-y).\xi}a(s(g),g^{-1}x,y,\xi)L_{g}^{-1}h(y)\,dy\,d\xi    \\
&=& (2\pi)^{-k}g\iint e^{\imath(g^{-1}x-g^{-1}z).\xi}a(s(g),
g^{-1}x,g^{-1}z,\xi)g^{-1}h(z)\left| J_{g^{-1}}(z)\right|\,dz\,d\xi    \\
&=& (2\pi)^{-k}\iint e^{\imath(x-z).\xi}b(g,x,z,\xi)
h(z)\,dz\,d\xi
\eeqns
where (cf. \cite[(4.10)]{Shubin})
\[  b(g,x,z,\xi)
=[ga(s(g),g^{-1}x,g^{-1}z,\psi(g,x,z)\xi)g^{-1}]
\left|\det\psi_{g}(x,z)\right|\left| J_{g^{-1}}(z)\right|  \]
and $g\to \psi(g,x,z)$ is $C^{\infty,0}$.
Then $g\to b(g,x,z,\xi)$ is a matrix-valued function with $S^{m}_{\rho,\de}$ entries, and it follows that  
$\widetilde{P}\in L_{\rho,\de}^{m}(r^{*}X;r^{*}E,r^{*}F)$.  It remains to show that $\widetilde{P}$ is proper.  To this end, by (\ref{eq:kgdist}),
the support of $L_{g}P^{s(g)}L_{g^{-1}}$ is $g\text{ supp }K^{s(g)}$.  So the support of 
$\widetilde{P}$ is the closure in $G\x X\x X$ of 
$\{(g,a,b): g^{-1}(a,b)\in \text{supp }K^{s(g)}\}$.  The properness of $\widetilde{P}$ now follows from that of $P$.
\end{proof}

\begin{definition}    \label{def:invpsdo}
A pseudodifferential family $D$ on $X$ is called {\em invariant} if
\begin{equation}  \label{eq:rgd}
L_{g}D^{s(g)}L_{g^{-1}}=D^{r(g)}
\end{equation}
for all $g\in G$.
\end{definition}

For each $u\in G^{0}$, $C_{c}(X^{u}, \mathcal{E}^{u})$ is a pre-Hilbert space where 
\[  \lan f,h\ran^{u}=\int \lan f(x),h(x)\ran^{u}\,d\mu^{u}.  \]
The Hilbert space completion of $C_{c}^{\infty}(X^{u},\mathcal{E}^{u})$ is denoted by 
$L^{2}(X^{u},\mathcal{E}^{u})$.  For each $g\in G$, the map $L_{g}$ extends to a unitary map
$L_{g}:L^{2}(X^{s(g)},\mathcal{E}^{s(g)})\to L^{2}(X^{r(g)},\mathcal{E}^{r(g)})$.  

The set of invariant, proper, pseudodifferential families is denoted by $\Psi^{m}_{\rho,\de}(X;E,F)$.  
Let $D\in \Psi^{m}_{\rho,\de}(X;E,F)$.   Then $D^{*}=\{(D^{u})^{*}\}\in \Psi^{m}_{\rho,\de}(X;F,E)$: it is invariant, since for
$f\in C^{\infty}_{c}(X^{r(g)}, \mathcal{E}^{r(g)}), h\in C^{\infty}_{c}(X^{s(g)}, \mathcal{F}^{s(g)})$,
\[ \lan f,L_{g}(D^{s(g)})^{*}L_{g^{-1}}h\ran^{r(g)}
=\lan L_{g}D^{s(g)}L_{g^{-1}}f,h\ran^{r(g)}
=\lan D^{r(g)}f,h\ran^{r(g)}=
\lan f,(D^{r(g)})^{*}h\ran^{r(g)}.       \]   

The condition that $supp\: D$ be compact for an invariant pseudodifferential family $D$
is too stringent. Instead, we will require that $\text{supp }D$ be {\em $G$-compact in $X*X$}. Connes (\cite[p.125]{Cointeg}), in the context of the holonomy groupoid, requires that the distribution $k$ have compact support where $K(x,y)=k(x^{-1}y)$.  A similar requirement applies in the homogeneous case of Connes and Moscovici (\cite[p.294]{CoMos}).  In their context (\cite[Definition 8]{NWX}), 
Nistor, Weinstein and Xu call a pseudodifferential family $D$ {\em uniformly supported} if its {\em reduced
support} $\mu_{1}(supp\: D)$ is compact where $\mu_{1}(x,y)=xy^{-1}$. 
In the special cases of these papers, the above conditions on the support of $K$ are each  equivalent to the condition used in this paper, viz. that $\text{supp }D$ be $G$-compact.

We now show that if $P$ is a pseudodifferential family with {\em compact} support
then it can be averaged over $G$ to give an invariant pseudodifferential
family. The converse is also true when $X$ is $G$-compact. These results for a pseudodifferential operator in the  homogeneous case are proved by Connes and Moscovici (\cite{CoMos}).  

\begin{proposition}     \label{prop:av}
Let $P\in L_{\rho,\de}^{m}(X;E,F)$ be compactly supported.  Then $\text{Av}(P)\in 
\Psi^{m}_{\rho,\de}(X;E,F)$ where
\begin{equation}  \label{eq:av}
\text{Av}(P)h(x)=\int L_{g}P^{s(g)}L_{g^{-1}}h(x)\,d\la^{p(x)}(g).
\end{equation}
\end{proposition}
\begin{proof}
Let $B'=\text{supp } P$.  By hypothesis, $B'$ is compact in $X*X$.  
Let $x_{0}\in X$ and $C$ be a fixed compact neighborhood of $x_{0}$ in $X$. We can suppose that $C$ is included in the domain of a chart.  By the properness of the $G$-action on $X$, the set $A=\{(g,y)\in G*_{r}X: (g^{-1}y,y)\in p_{2}(B')\x C\}$ is compact.  For $u\in G^{0}$, let $A^{u}=p_{1}(A)\cap G^{u}$.  For each $g'\in A^{p(x_{0})}$,  there exist charts $Z_{g'}, U_{g'}, V_{g'}$ as in the proof of Proposition~\ref{prop:wdd} with 
$g'\in Z_{g'}, x_{0}\in U_{g'}\subset C^{0}$, the interior of $C$.  In terms of local coordinates on $U_{g'}$, the pdo family
$\widetilde{P}$ is given by continuous function $g\to b^{g'}(g,x,y,\xi)$ with values in 
$S^{m}_{\rho,\de}$.  Cover the compact set $A^{p(x_{0})}$ by a finite number $Z_{g_{1}}, \ldots ,Z_{g_{n}}$ of the 
$Z_{g'}$'s.  Let $U=\cap_{i=1}^{n}U_{g_{i}}$.
Then for some compact neighborhood $T$ of $r(x_{0})$ in $G^{0}$, by the compactness of $A$, we have $A^{u}\subset \cup_{i=1}^{n}Z_{g_{i}}=Z$ for all $u\in T$.  By contracting $U$, we
can suppose that 
$p(U)\subset T$ and $U$ is the domain of a chart of $X$.  We now show that on $U$, $\text{Av}(P)$ is a pdo family.  Let $B=\cup_{u\in T}A^{u}\subset Z$.  Then $B$ is compact, and there exist
$\phi_{i}\in C_{c}^{\infty,0}(Z_{g_{i}})$ such that $\sum_{i=1}^{n}\phi_{i}=1$ on $B$.  Then 
for $x\in U$, $h\in C_{c}^{\infty,0}(U,E_{U})$ and in terms of local coordinates,
\beqns
\text{Av}(P)(h)(x)&=&\iint_{A} gK(g^{-1}x,g^{-1}y)g^{-1}h(y)\,(d\la^{p(x)}\x d\mu^{p(x)})(g,y)\\
&=& \sum_{i=1}^{n}\int \phi_{i}(g)L_{g}P^{s(g)}L_{g^{-1}}h(x)\,d\la^{p(x)}(g)\\
&=& (2\pi)^{-k}\iint e^{\imath(x-y).\xi}a(x,y,\xi)h(y)\,dy\,d\xi
\eeqns
where in an obvious notation (cf. Proposition~\ref{prop:wdd})
\[ a(x,y,\xi)=\sum_{i=1}^{n}\int\phi_{i}(g)b_{i}(g,x,y,\xi)\,d\la^{p(x)}(g)\in S^{m}_{\rho,\de}.  \]
It is easy to check properness and invariance for $\text{Av}(P)$.
\end{proof}

\begin{proposition}          \label{prop:pavf}
Let $D\in \Psi^{m}_{\rho,\de}(X;E,F)$, $X$ be $G$-compact and $c$ be a cut-off function for 
$X$ (Proposition~\ref{prop:cgx}).  Then the family $P=cD$ has compact support and 
$D=\text{Av}(P)$.
\end{proposition}
\begin{proof}
The argument of \cite[Lemma 1.2]{CoMos} gives the result.
\end{proof}   

For the theory presented here, it is crucial that $X$ be $G$-compact.  Indeed, if not, then there are no invariant, elliptic families of pdo's on $X$ with $G$-compact support.
Let $X$ be $G$-compact and write $S^{m}_{\rho,\de}=S^{m}_{\rho,\de}(T^{*}X;\text{ Hom }(\pi^{*}E,\pi^{*}F))$.  
The principal symbol for an invariant family $D$ can be usefully expressed in terms of invariant symbols as follows.

An element $a\in S_{\rho,\de}^{m}$ is called {\em invariant} if $g.a^{s(g)}=a^{r(g)}$ for all $g\in G$,
i.e. if for all $g\in G$, $x\in X^{r(g)}$ and $(x,\xi)\in T^{*}X$, we have 
$ga(g^{-1}x,g^{-1}\xi)g^{-1}=a(x,\xi)$ where $g^{-1}\xi=(^{t}\ell_{g^{-1}}')^{-1}\xi$.
Let $S^{m,i}_{\rho,\de}$ be the set of invariant elements of $S^{m}_{\rho,\de}$.  Let $D\in \Psi^{m}_{\rho,\de}(X;E,F)$.  Write $D=\text{Av}(P)$ as in Proposition~\ref{prop:pavf}.
A partition of unity argument shows that there exists 
$p(x,\xi)\in S^{m}_{\rho,\de}$, compactly supported in $x$, such that locally, 
$(\si_{P}-p)\in S^{m-2(\rho -1/2)}_{\rho,\de}$.  Now locally, using (\ref{eq:*}), we have 
that $\si_{\widetilde{P}} - \tilde{p}\in S_{\rho,\de}^{m-2(\rho-1/2)}$ where 
$\tilde{p}^{g}=g.p^{s(g)}$.  It is left to the reader to check that 
$a=\text{Av}(p)\in S^{m,i}_{\rho,\de}$, where 
$\text{Av}(p)=\int \tilde{p}^{g}(x,\xi)\,d\la^{p(x)}(g)$, and locally, 
$(\si_{D}-a)\in S^{m-2(\rho -1/2)}_{\rho,\de}$.  Define
\[  \si(D)=[a]=a+S^{m-2(\rho -1/2),i}_{\rho,\de}\in 
S^{m,i}_{\rho,\de}/S^{m-2(\rho -1/2),i}_{\rho,\de}. \]  
Then $\si(D)$ is independent of the choice of $a$.  The class $\si(D)$ is called the {\em principal symbol} of $D$.  (This determines the principal symbol of $D$ in  
$S^{m}_{\rho,\de}/S^{m-2(\rho -1/2)}_{\rho,\de}$ discussed earlier.)

Every element of $S^{m,i}_{\rho,\de}/S^{m-2(\rho -1/2),i}_{\rho,\de}$ is the principal symbol 
of some $D\in \Psi^{m}_{\rho,\de}(X;E,F)$.  Indeed, given $a\in S^{m,i}_{\rho,\de}$ we let
$p(x,\xi)=c(x)a(x,\xi)$, and $P\in L^{m}_{\rho,\de}(X;E,F)$ where $P$ has compactly supported kernel and locally, $(\si_{P}-p)\in S^{m-2(\rho -1/2)}_{\rho,\de}$.  Then $\si(D)=[a]$ where
$D=\text{Av}(P)$.  So we have the following theorem.

\begin{theorem}        \label{th:si}
Let $X$ be $G$-compact.  Then the map $\si:\Psi^{m}_{\rho,\de}(X;E,F)\to S^{m,i}_{\rho,\de}/S^{m-2(\rho -1/2),i}_{\rho,\de}$
is onto.
\end{theorem} 


\section{The equivariant analytic index} 

\newcommand{\bga}{\textbf{B}(\gah)}
\newcommand{\kga}{\textbf{K}(\gah)}
\newcommand{\bmfH}{\textbf{B}(\mfH)}
\newcommand{\kmfH}{\textbf{K}(\mfH)}
\newcommand{\bgga}{\textbf{B}_{G}(\mfH)}
\newcommand{\kgga}{\textbf{K}_{G}(\mfH)}
\newcommand{\gach}{\Ga_{c}(\mfH)}
\newcommand{\gah}{\Ga(\mfH)}

We now discuss ellipticity for a pseudodifferential family $D$.  The notion is an adaptation of that given by Shubin (\cite[\S 5]{Shubin}).  Let $Y=T\x W$ and $D\in 
L^{m}_{\rho,\de}(Y;Y\x \C^{p},Y\x \C^{q})$.  We say that $D$ is {\em elliptic} if there exists a properly supported $D'\in L^{m}_{\rho,\de}(Y;Y\x \C^{p},Y\x \C^{q})$
and $R'\in L^{-\infty}(Y;Y\x \C^{p},Y\x \C^{q})$ such that $D=D'+R'$ and for arbitrary $C\in \mathcal{C}(Y)$, there exist
positive $R, C_{1}, C_{2}$ such that for
$\left|\xi\right|\geq R$ and $x\in C$, we have 
\begin{equation}         \label{eq:ellpdo}
C_{1}\left|\xi\right|^{m}\leq \left|\si_{D'}(x,\xi)\right|\leq C_{2}\left|\xi\right|^{m}. 
\end{equation}
More generally,  
$D\in L^{m}_{\rho,\de}(X;E,F)$ is called {\em elliptic} if locally, every $(D_{U})_{1}$ is elliptic.
By the continuous version of \cite[Proposition 5.5]{Shubin},
elliptic operators are invariant with respect to change of variables. 

For the rest of the paper, $G$ is a continuous family groupoid acting properly on a $G$-compact, $C^{\infty,0}$ $G$-space $(X,p)$. Further, $\{\la^{u}\}$ is a $C^{\infty,0}$ left Haar system on $G$ and $\{\mu^{u}\}$ is a $C^{\infty,0}$ left Haar system on $X$.  Also, $E, F$ are $C^{\infty,0}$ $G$-vector bundles over $X$, and 
$u\to \left< , \right>^{u}$ stands for a $G$-isometric metric on each of $E, F$.

\begin{theorem} \label{th:para}
Let $D$ be an elliptic element of $\Psi^{m}_{\rho,\de}(X;E,F)$. 
Then there exists an elliptic family $Q\in \Psi_{\rho,\de}^{-m}(X;F,E)$ such that 
\[  QD=I+R_{1} \hspace{.2in} DQ=I+R_{2} \]
where $R_{1}, R_{2}$ are $\Psi^{-\infty}$-pseudodifferential families.
\end{theorem}
\begin{proof}
One follows the proof of \cite[Theorem 5.1]{Shubin}) to obtain a ``continuous family'' parametrix $Q'$ and $L^{-\infty}$-families $R_{1}', R_{2}'$ such that
$Q'D=I+R_{1}', DQ'=I+R_{2}'$.  Then following the argument of the proof of \cite[Proposition 1.3]{CoMos}, we can take $Q=\text{Av}(cQ')$ where $c$ is a cut-off function.  
\end{proof}

We now discuss the boundedness of invariant families.  To this end, we want to reduce our considerations to the case where $m=0$. We consider the invariant symbol
$d$ given by:
\[   d(x,\xi)=(1+\norm{\xi}^{2})^{-m/2}                  \]
where $\norm{.}$ is the norm function on $T^{*}X$ determined by an invariant, $C^{\infty,0}$,
hermitian, metric.  (This gives a non-trivial example of a $G$-invariant elliptic pseudodifferential family on $X$.  Other examples will be given in \S 6.)
Let $D\in \Psi^{m}_{\rho,\de}(X;E,F)$ and
$a$ be an invariant symbol for $D$.  Then $ad\in S^{0,i}_{\rho,\de}$.  By Theorem~\ref{th:si}, the principal symbol map $\si$ is onto.  So there exists an element
$L\in \Psi^{0}_{\rho,\de}(X;E,F)$ whose principal symbol is $[ad]$.
We then define (cf. \cite{Kasp1}) the analytic index of $D$ to be that of $L$, and so can suppose from the outset that $D\in \Psi_{\rho,\de}^{0}(X;E,F)$.  (Note that the choice of $L$ does not matter since, as in the families case of Atiyah and Singer (\cite[p.127]{AS4}), the analytic index depends only on the symbol class (Proposition~\ref{prop:indsymb}).)

\begin{proposition}        \label{prop:cont1}
Let $W$ be an open subset of $\R^{k}$ and $A\in 
L^{0}_{\rho,\de}(W;W\x \C^{p},W\x \C^{q})$ where $A$ is compactly supported.  Then $A$ extends to a bounded linear operator $A'$, from $L^{2}(W,\C^{p})$ to $L^{2}(W,\C^{q})$, and the map $A\to A'$ is continuous.  Further, if 
in addition $A\in L^{m}_{\rho,\de}(W;W\x \C^{p},W\x \C^{q})$ where $m<0$, then $A'$ is compact.
\end{proposition}
\begin{proof}
The first part  follows from the ``very simple proof of $L^{2}$ continuity'' given by H\"{o}rmander in \cite[p.75]{Hor}.  The second part of the proposition follows from 
\cite[Corollary 6.1]{Shubin}.
\end{proof}
\begin{corollary}    \label{cor:Abdd}
Let $W$ be an open subset of $\R^{k}$, $Y=T\x W$, and $B\in L_{\rho,\de}^{0}(Y;Y\x \C^{p},Y\x \C^{q})$ have compact support $C\subset T\x W\x W$. Then the map $u\to (B^{u})'$ is continuous, and  $\sup_{u\in T}\norm{(B^{u})'}<\infty$.
\end{corollary}

Let $A\in L^{0}_{\rho,\de}(X;E,F)$.  For each $u\in G^{0}$, the map 
$A^{u}:C_{c}^{\infty}(X^{u},\mathcal{E}^{u})\to  C_{c}^{\infty}(X^{u},\mathcal{F}^{u})$.
We say that $A$ is {\em bounded} if each $A^{u}$ extends to a bounded linear operator, also denoted by $A^{u}: L^{2}(X^{u}, \mathcal{E}^{u})\to  L^{2}(X^{u}, \mathcal{F}^{u})$, and $\norm{A}=\sup_{u\in G^{0}}\norm{A^{u}}<\infty$.

\begin{proposition}     \label{prop:pbdd}
Let $P\in L_{\rho,\de}^{0}(X;E,F)$ be compactly supported.  Then $P$ is bounded.
\end{proposition}
\begin{proof} 
The kernel $K$ of $P$ has compact support $C'$ in $X*X$.  Cover $C'$ by a finite number $n$ of sets $U_{i}*U_{i}$ where $U_{i}$ is a chart in $X$ that trivializes both $E,F$.  Let $\phi_{i}\in 
C_{c}^{\infty,0}(U_{i}*U_{i})$ be such that $0\leq \phi_{i}\leq 1$ and $\sum_{i=1}^{n}\phi_{i}=1$ on $C'$.  Then $P$ is the sum of $P_{i}$'s, where the kernel of 
$P_{i}$ is $\phi_{i}K$, and we can suppose that $P$ is a $P_{i}$ and $U=U_{i}$.
To relate the norm of $P$ to that of its local coordinate version $(P_{U})_{1}$, we need to switch from $U, E_{U}, F_{U}$ to $p(U)\x W, p(U)\x W\x \C^{p}, p(U)\x W\x \C^{q}$.  To this end, we use 
(\ref{eq:du1}) and the Radon-Nikodym derivative $d(\mu^{u}\circ \phi^{u})/d\la^{W}(u,w)$ of 
Definition 4, (ii).  The main point here is that the maps involved (such as $\al$) only need to be considered on a compact set, where they are both bounded and bounded away from $0$.  So $P$ is bounded if and only if $(P_{U})_{1}$ is bounded.  But $(P_{U})_{1}$ is bounded by 
Corollary~\ref{cor:Abdd}.
\end{proof}

For the rest of the paper, $D$ will be an elliptic element of $\Psi^{0}_{\rho,\de}(X;E,F)$.

\begin{lemma}   \label{lemma:gbdd}
Let $f\in C_{c}(X)$.  Then 
$\sup_{x\in X}\left|\int f(g^{-1}x)\,d\la^{p(x)}(g)\right|=M_{f}<\infty$.
\end{lemma}
\begin{proof}
Since $X$ is $G$-compact, there exists a compact subset $C$ of $X$ such that $X=GC$.  Let $Z=\text{ supp }f$.  Then $Z$ is compact.  By the properness of the action, the set $A$ of elements $g\in G$ such that 
$g^{-1}y\in Z$ for some $y\in C$ is compact.  So 
$\sup_{y\in C}\left|\int f(g^{-1}y)\,d\la^{p(y)}(g)\right|=M_{f} <\infty$.  Let $x\in X$.  Write $x=hy$ for some
$y\in C, h\in G$.  Then $\int f(g^{-1}x)\,d\la^{p(x)}(g)=\int f(g^{-1}y)\,d\la^{p(y)}(g)$, and the result follows.
\end{proof}

\begin{theorem}       \label{th:gbdd}
$D$ is bounded.
\end{theorem}
\begin{proof}
We adapt the argument of Connes and Moscovici (\cite[pp.296-297]{CoMos}) to show that $D$ is bounded.  To this end, write
$D=\text{Av}(P)$ where $P$ is as in Proposition~\ref{prop:pavf}.  Since $P$ has compact support, it is bounded (by Proposition~\ref{prop:pbdd}).   Let  
$k\in C_{c}^{\infty,0}(X,E)$.  Define $F\in 
C_{c}^{\infty,0}(r^{*}X,r^{*}F)$ by: $F(g,x)=L_{g}P^{s(g)}L_{g^{-1}}k(x)$. 

Let $v\in G^{0}$ and $g,h\in G^{v}$.  Let $f\in C_{c}^{\infty,0}(X)$ be such that $f=1$ on $p_{2}(C')$ where $C'$ is the support of the kernel $K$ of $P$.  So $Pf=P$, and 
using Lemma~\ref{lemma:gbdd} and the argument of \cite[p.297]{CoMos}, there exists 
$M=M_{f^{2}}$ such that 
\[  \int\norm{F_{g}}^{2}\,d\la^{v}(g)\leq M\norm{P}^{2} \norm{k^{v}}_{2}^{2}.   \]

Now suppose that 
$\lan F_{g},F_{h}\ran^{v}=\int \lan F_{g}(x),F_{h}(x)\ran\,d\mu^{v}(x)$ is non-zero.  
The expression for the kernel of $L_{g}P^{s(g)}L_{g^{-1}}$ shows that for some $x$, both 
$g^{-1}x, h^{-1}x\in p_{1}(C')$. 
The properness of the $G$-action gives that there exists a compact subset $C_{1}$ of $G$ such that $\lan F_{g}, F_{h}\ran=0$ whenever $g^{-1}h\notin C_{1}$.  Let $\phi\in C_{c}(G)$ be such that $\phi=1$ on $C_{1}$.  The argument of \cite[Lemma 1.5]{CoMos} then gives,
with $\psi(g)=\norm{F_{g}}$, that
\begin{gather*}
\norm{D^{v}k^{v}}_{2}^{2}=\norm{\text{Av}(P)^{v}k^{v}}_{2}^{2}=
\norm{\int F_{g}\,d\la^{v}(g)}_{2}^{2}\\
\leq
(\int\norm{F_{g}}^{2}\,d\la^{v}(g))^{1/2}\norm{[R_{\widehat{\phi}}(\psi)]^{v}}_{2}
\leq M\norm{P}^{2}\norm{R_{\widehat{\phi}}}\norm{k^{v}}_{2}^{2}.   
\end{gather*}
(Here, $R_{\phi}F=F*\phi$ as in \S 3, and $\widehat{\phi}(x)=\phi(x^{-1})$.)
 So $D$ is bounded.
\end{proof}
\begin{corollary}   \label{cor:PD}
Let $U$ be a relatively compact, open subset of $X$.  Let $\mathcal{P}_{U}$ be the set of $P$'s 
with compact support contained in $U*U$.  Then the map $P\to \text{Av}(P)$ is norm 
continuous on $\mathcal{P}_{U}$.
\end{corollary}
\begin{proof}
The functions $f, \phi$ of the preceding proof can be taken to depend only on $U$, so that there 
exists a constant $M'$ such that $\norm{\text{Av}(P)}\leq M'\norm{P}$.
\end{proof}


We now discuss the Kasparov modules that will give us the analytic index of $D$.  The approach is similar to that of \cite{jsrc} except that the groupoid version of the equivariant K-theory of N. C. Phillips is not used.

\begin{proposition}       \label{prop:preH}
The space $C_{c}(X,E)$ is a pre-Hilbert module over
the pre-$C^{*}$-algebra $C_{c}(G)\subset C_{red}^{*}(G)$ with $C_{c}(G)$-inner product and module
action given by:
\beqn
\lan k_{1},k_{2}\ran(g)&=&\int_{X^{r(g)}}\ov{\lan k_{1}(x),L_{g}k_{2}(x)}\ran\,
d\mu^{r(g)}(x)    \label{eq:lanran2}      \\
kf(x)&=&\int_{G^{p(x)}}L_{g}k(x)f(g^{-1})\,d\la^{p(x)}(g).   \label{eq:efx4}
\eeqn
\end{proposition}
\begin{proof}
It is easy to prove that $\lan k_{1},k_{2}\ran\in C_{c}(G)$.  The rest of the proof follows that of \cite[Proposition 5.6]{jsrc}, using the first approach to $C_{red}^{*}(G)$ in \S 3.
\end{proof}

It follows that the completion $\Ga(E)$ of $C_{c}(X,E)$ under the norm 
$k\to \norm{\lan k,k\ran}^{1/2}$ is a Hilbert $C_{red}^{*}(G)$-module.

\begin{proposition}   \label{prop:dinl}
The pdo family $D$ extends by continuity to an element (also denoted $D$) of $\mfL(\Ga(E),\Ga(F))$.  Further, if 
$R\in \Psi^{m}_{\rho,\de}(X;E,F)$ where $m<0$, then $R$ extends to an element of  $\mfK(\Ga(E),\Ga(F))$.
\end{proposition}
\begin{proof}
Let $M=\sup_{u}\norm{D^{u}}$.  By Theorem~\ref{th:gbdd}, $M<\infty$.  The argument of 
\cite[Theorem 5.7]{jsrc}, which uses the fact that  
$D^{*}\in \Psi^{0}_{\rho,\de}(X;F,E)$, then gives that $D\in \mfL(\Ga(E),\Ga(F))$ and $\norm{D}\leq M$. 
 
Now let $R\in \Psi^{m}_{\rho,\de}(X;E,F)$ where $m<0$.  Then let $R=\text{Av}(P)$ where 
$P\in L^{m}_{\rho,\de}(X;E,F)$ with compactly supported kernel $K$.  As in the proof of 
Proposition~\ref{prop:pbdd}, we can suppose that the support $K$ of $P$
is contained in $U*U$, where 
$U\thicksim p(U)\x W$ is a relatively compact chart of $X$
that trivializes $E$ and $F$.  By Proposition~\ref{prop:cont1}, the map 
$u\to (D_{U})_{1}^{u}$ 
((\ref{eq:du1})) is continuous into $K(L^{2}(W,\C^{p}), L^{2}(W,\C^{q}))$. For 
$\xi\in C_{c}^{\infty,0}(U,F_{U}), \eta\in C_{c}^{\infty,0}(U,E_{U})$ let 
$\xi\otimes\eta: C_{c}^{\infty,0}(X,E)\to C_{c}^{\infty,0}(X,F)$ be given by:
\[   \xi\otimes \eta(f)=\xi\ov{\lan\eta,f\ran'}.  \]
(Here, $\lan\eta,f\ran'$ is the usual inner product on $L^{2}(E_{U})$ - we use the $'$ here to distinguish this inner product from that of (\ref{eq:lanran2}).)
It is easily checked that $\xi\otimes \eta\in 
L_{\rho,\de}^{-\infty}(X;E,F)$ and is compactly supported.  Let $\eps>0$.
Using a partition of unity argument and the density of the finite rank operators in $K(L^{2}(W,\C^{p}), L^{2}(W,\C^{q}))$, there exist $\xi_{i}\in C_{c}^{\infty,0}(U,F_{U}), \eta_{i}\in C_{c}^{\infty,0}(U,E_{U})$
such that $\norm{P-S}<\eps$ where $S=\sum_{i=1}^{n}\xi_{i}\otimes \eta_{i}$.  By Corollary~\ref{cor:PD}, there exists
$M'>0$ dependent only on $U$ such that $\norm{R-\text{Av}(S)}<M'\eps$.  Now 
we claim that
$\text{Av}(\xi\otimes \eta)=\theta_{\xi,\eta}$ so that $R\in \mfK(\Ga(E),\Ga(F))$.

To prove this equality (cf. \cite{jsrc}):
\begin{align*}
\text{Av}(\xi\otimes \eta)f(x) &= 
\int L_{g}(\xi\otimes\eta)L_{g^{-1}}f(x)\,d\la^{p(x)}(g)\\
&= \int L_{g}[\xi\ov{\lan\eta,L_{g^{-1}}f\ran'}](x)\,d\la^{p(x)}(g)\\
&= \int g[\xi\ov{\lan\eta,L_{g^{-1}}f\ran'}](g^{-1}x)\,d\la^{p(x)}(g)\\
&= \int g\xi(g^{-1}x)\ov{\lan\eta,L_{g^{-1}}f\ran'}(s(g))\,d\la^{p(x)}(g)\\
&= \int L_{g}\xi(x)\,d\la^{p(x)}(g)
\int\ov{\lan\eta(y),L_{g^{-1}}f(y)\ran}\,d\mu^{s(g)}(y)\\
&= \int L_{g}\xi(x)\lan\eta,f\ran(g^{-1})\,d\la^{p(x)}(g)\\
&= (\xi\lan\eta,f\ran)(x)\\
&= \theta_{\xi,\eta}f(x).
\end{align*}
\end{proof}

We now specify how $D$ determines the analytic index $ind_{a}(D)\in K_{0}(C_{red}^{*}(G))$.  To this end, let $Q$ be a parametrix of $D$ as in Theorem~\ref{th:para} and let $\Ga$ be the (ungraded)
Hilbert module $\Ga(E)\oplus \Ga(F)$ over $C_{red}^{*}(G)$.  Let $T\in \mfL(\Ga)$ be defined by:
\[
T=\normalsize
\begin{pmatrix}
0 & Q \\
D & 0
\end{pmatrix}
\large
\]
By Theorem~\ref{th:para} and Proposition~\ref{prop:dinl}, $T\in \mfL(\Ga)$ and
$(T^{2}-I)\in \mfK(\Ga)$.  Let $c\in \mfL(\Ga)$ be such that the image of $c$ in  
$\mfL(\Ga)/\mfK(\Ga)$ is the unitary part of the image of $T$.  Then we
get a Kasparov $(\C,C_{red}^{*}(G))$-bimodule
\[ (\Ga\oplus \Ga,
\normalsize
\begin{pmatrix}
0 & c^{*} \\
c & 0
\end{pmatrix}
\large
)
\]
which gives an element of $KK^{0}(\C, C_{red}^{*}(G))=
K_{0}(C_{red}^{*}(G))$.  This element is the analytic index $\text{ind}_{a}(D)$ of $D$.

\begin{proposition}       \label{prop:indsymb}
The analytic index of $D$ depends only on the symbol class $\si(D)$.
\end{proposition}
\begin{proof}
Let $D'$ be elliptic and such that $\si(D')=\si(D)$.  By definition of the principal symbol, 
$D-D'\in \Psi^{-2(\rho -1/2)}_{\rho,\de}$, and by Proposition~\ref{prop:dinl}, belongs to
$\mfK(\Ga(E),\Ga(F))$.  Then $Q,Q'$ have the same principal symbol class, and so $Q-Q'\in \mfK(\Ga(F),\Ga(E))$ and
$T'$ is equivalent to $T$.  So $D, D'$ have the same analytic index.
\end{proof} 


\section{Comments and examples}

\noindent
{\bf (1) The case where $X=G$.}\\

This is the case that has been most studied in the literature.  Generally, $G$ is assumed to be a Lie groupoid (sometimes with corners) and the pdo's classical.   We note first that the action of $G$ on $X=G$ (left multiplication) is proper so that the theory developed here applies to that situation.  In the case where $X=G$, the invariance of the symbol gives that we need only consider it on the Lie algebroid $A(G)=\cup_{u\in G^{0}}T_{u}G\subset TG$ of $G$.  In the work of Nistor, Weinstein and Xu (\cite{NWX}) (see also the paper \cite{Mont} of Monthubert), the symbol calculus is developed in the context of the Poisson algebra of functions on the dual of $A(G)$. In the general case of the present paper, there is no Lie algebroid available.  So we have to work with $TX$ rather than $A(G)$, and a different approach is needed.   

In \cite{NWX}, a number of examples are given illustrating the theory.  In particular, by taking $G=M\x M$, they show how the classical (non-equivariant) Atiyah-Singer index theorem fits in with the theory.  In the Atiyah-Singer (and the present) context, the theorem is, of course, obtained by taking $G$ to be the trivial one-element group acting on $M$ (so that $M$ is regarded as a family over a one point space.)  They further consider the case of the holonomy groupoid associated with the foliation determined by a locally free action of a Lie group on a manifold $M$.  (An isotropy condition on the action is also required.) The case of a general $C^{\infty,0}$-foliation will be discussed below.  They also consider the {\em adiabatic groupoid} associated with a Lie groupoid, an example of which is the tangent groupoid as defined by Connes (in which case, the Lie groupoid is $M\x M$).  This construction remains valid (\cite{LMN00}) if the Lie groupoid is replaced by a continuous family groupoid.  

Pseudodifferential analysis on continuous family groupoids is studied by Lauter, Monthubert and Nisor in \cite{LMN00}.  As described above, the Lie algebroid plays a fundamental role.  As far as the class of pdo families are concerned, the classical pdo's of \cite{LMN00} are, of course, in the $L^{m}_{1,0}$ class of the present paper, but the pdo's in the latter class are more general - the asymptotic expansions need not be in terms of homogeneous symbols.  (As H\"{o}rmander points out (\cite[p.65]{Hor}), there are advantages in not insisting on homogeneity.)  Allowing $L^{m}_{\rho,\de}$ operators for general $\rho,\de$, of course, widens even more the class of pdo families considered. (See (5) below.)  Because of this generality (as well as the lack of an $A(G)$), in the present paper, one has to give up Theorem 1 of \cite{LMN00} since homogeneity is needed in order to involve the cosphere bundle.  

Assuming $G^{0}$ compact, the authors show that the principal symbol of an invariant elliptic pdo family defines a principal symbol class $[\si_{m}(D)]\in K^{0}(A^{*}(G))$. The analytic index morphism $\text{ind}_{a}:K^{i}(A^{*}(G))\to K_{i}(C^{*}(G))$ is constructed using the adiabatic groupoid, and $\text{ind}_{a}([\si_{m}(D)])$ is the analytic index of $D$.  In the situation of the present paper, the principal symbol of $D$ is an equivalence class of $S_{\rho,\de}^{m}(\text{Hom}(\pi^{*}E,\pi^{*}F))$, and defines through the constructed Kasparov module an element of $K(C^{*}_{red}(G))$.  The Kasparov module approach is a generalization of the original way that Atiyah and Singer obtained the analytic index $\text{ind}_{a}(D)$ directly as the Fredholm index of $D$.

Properness plays a key role in the present paper. An advantage of working in the $X=G$ context of \cite{LMN00} is that (following Connes \cite{Cointeg})
one can use $k(g(g')^{-1})$, where $k$ is a distribution on $G$, when discussing the kernel $K(g,g')$ of the pdo family. In the proof of boundedness, \cite{LMN00} adapts Theorem 18.1.11 of \cite{Hor}, while the present paper uses ``the very simple proof of $L^{2}$ continuity'' of \cite{Hor} immediately following that theorem.  Other parts of the proof of the boundedness theorem of \cite{LMN00} are not available in the generality of the present paper.  Instead, the approach of Connes-Moscovici, using cut-off functions and $\text{Av}(P)$, was modified. 

The study of the case $X=G$ has (as above) been remarkably successful and one might wonder if the general case can be reduced to that by changing the groupoid.  An example of this would be the use of the groupoid $M\x M$ in \cite{NWX} (see above) to interpret in groupoid terms the classical, non-equivariant index theorem for $M$.  But I do not know of a similar interpretation for the equivariant case.  (See (2) below.)

Such a reduction to the case $X=G$ also does not apply to the families theorems.  The natural groupoid interpretation (\cite{families}) is in terms of $G$-spaces where $X\ne G$.  For the original families theorem (\cite{AS4}), we have a fiber bundle $(X,p)$ over a compact space $T$ with $C^{\infty}$ compact manifold fiber.   
In this case, the continuous family groupoid involved is just
$T$, regarded as a unit space groupoid, and the action of $T$ on $X$ is given by: $t.x=x$ for $x\in X^{t}$.  More interesting, for the Atiyah-Singer equivariant families index theorem ($G$ a compact Lie group), the groupoid involved is the transformation group $G\x T$ which acts on $X$ by: $(g,t)x=gx$ for $x\in X^{t}$.  For these and other index theorems, we {\em have} to consider the continuous family case in a situation where $X\ne G$.  The case of a Lie groupoid (or even a continuous family groupoid) acting on itself is not enough. 

Another very important case - in the smooth category - where we need $X\ne G$ occurs in Connes's construction of the geometrical cycles for a smooth groupoid $G$ and the analytic assembly map briefly discussed below.  This involves a proper action of $G$ on a smooth manifold $P$ - $P$, of course, is {\em a fortiori} a $C^{\infty,0}$ $G$-space in the sense of this paper, and $P\ne G$ in general. The theory of this paper applies to that situation.  The $G$-space notion also arises very naturally in other contexts, e.g. in Morita theory for groupoids (\cite{MuReW}).\\

\noindent
{\bf (2) Index theorems for group actions}\\

Connes and Moscovici (\cite{CoMos}), against the background of the $L^{2}$-index theorem for covering spaces of Atiyah and Singer, proved the $L^{2}$-index theorem for homogeneous spaces of Lie groups.  More precisely, let $G$ be a Lie group (with additional, natural, conditions) and $H$ be a compact subgroup of $G$.  One then takes $X=G/H$ with $G$ acting by left multiplication on the cosets.  Obviously this action is proper (and $X\ne G$), and the results of the present paper apply.  Connes and Moscovici proved a numerical valued index theorem for a $G$-invariant elliptic pdo $D$ of order $0$ on $G/H$ (homogeneous vector bundles).  The analytic index $\text{ind}_{G}D$ of $D$ is defined using the $G$-trace $\text{tr}_{G}$.  Also $D$ defines 
through its principal symbol an element $[\si_{0}(D)]$ of the equivariant K-group $K_{H}(V)$, where  
$V=\{\xi\in \mathfrak{g}^{*}: \xi\mid \mfh=0\}$ (and $\mfg, \mfh$ are the Lie algebras of $G, H$ respectively).  They constructed a topological index $\text{ind}_{t}:K_{H}(V)\to \R$, and   using the heat equation method, showed that $\text{ind}_{G}D=\text{ind}_{t}([\si_{0}(D)])$. 

An analytic index in $K(C^{*}(G))$ was constructed later in great generality by Connes in his book (\cite[p.136f.]{Connesbook}), using the analytic assembly map and proper G-manifolds.  (This theory extends in the natural way to the continuous family context.) Connes constructs a map $\mu:K_{top}^{*}(G)\to K(C^{*}(G))$ which in turn determines in the obvious way a map
$\mu_{r}:K_{top}^{*}(G)\to K(C_{red}^{*}(G))$.  Here, 
$K_{top}^{*}(G)$ is the additive group determined by the geometric cycles.  Further in the  context of \cite{CoMos}, $K_{H}(V)=K(C^{*}(TP\rtimes G))\to K_{top}^{*}(G)$, so that we can regard
$[\si_{0}(D)]\in K_{top}^{*}(G)$.  He defines an analytic index for $D$ in $K(C^{*}(G))$ using quasi-isomorphisms, and shows that that index equals $\mu([\si_{0}(D)])$.  It seems very likely that
$\mu_{r}([\si_{0}(D)])$ coincides with $\text{ind}_{a}(D)$ in the sense of the present paper.  

In his paper \cite{Kasp1}, G. G. Kasparov sketched a proof of a general equivariant index theorem for group actions.  (I understand that a paper giving the details of the proof has not yet appeared.)  In Kasparov's setting, we have a separable locally compact group $G$ acting properly, smoothly and isometrically on a connected, $G$-compact, complete, Riemannian manifold $X$.  We are also given 
a $G$-invariant elliptic pdo $D$ on $X$.
Kasparov constructs a Hilbert $C^{*}(G)$-module on which $D$ is a Fredholm operator, and this gives a class in $K(C^{*}(G))$.  This class is called the {\em analytic index} of $D$.  Our construction in 
Proposition~\ref{prop:preH} is a groupoid version of Kasparov's.  Kasparov uses equivariant KK-theory to construct the topological index. \\

\noindent
(3) {\bf Examples of continuous family groupoids}\\

One way to obtain continuous family groupoids is by considering closed subgroupoids of Lie groupoids.  Let $H$ be a Lie groupoid and $X$ be an invariant closed subset of $H^{0}$.  Then the reduction of $H$ to $X$ is a continuous family groupoid.  Transformation groups $G\x X$, where $G$ is a Lie group and $X$ a locally compact space, are also continuous family groupoids (\cite{families}).  Although one normally considers smooth actions of a Lie group $G$ on a manifold $X$ (so that $G\x X$ is a Lie groupoid), there are many examples where one just has a continuous action of $G$ on a locally compact space $X$ (so that $G\x X$ is not a Lie groupoid), e.g. there is always the trivial action of $G$ on any $X$.  More interesting, if $G$ acts on a compact space $Y$, then 
it also acts continuously on $X=\prod_{i=1}^{\infty}Y$ with the diagonal action.

The first continuous family groupoid (not a Lie groupoid) that seems to have been explicitly considered in the literature is the holonomy groupoid of a $C^{\infty,0}$-foliated manifold $(V,\mathcal{F})$.  This was investigated in detail by Connes in \cite[p.111f.]{Cointeg}.   In their exposition of the Connes index theorem of that paper, Moore and Schochet (\cite{MooSchoch}) realized that it could be extended to foliated spaces (so that $V$ is not assumed to be a manifold). Effectively, a foliated space is a separable metrizable space $V$ with an atlas of charts $U_{x}\sim 
L_{x}\x N_{x}$ where $L_{x}$ is open in $\R^{p}$ ($p$ fixed) and 
whose coordinate changes are $C^{\infty,0}$.  (This notion is related to, but not the same as, the notion of a continuous family $(X,p)$ of manifolds discussed in the present paper.)  There are many examples of foliated spaces $V$ which are not manifolds e.g. a solenoid (\cite[p.42]{MooSchoch}). 

The holonomy groupoid $G$ of a foliated space $V$ is a continuous family groupoid in the sense of this paper.  To see this, recall (e.g. \cite[2.3]{Paterson}) that the elements of $G$ are triples $(x,[\ga],y)$ where $x,y$ belong to a leaf $L$, $\ga$ is a path in $L$ from $x$ to $y$ and $[\ga]$ is the holonomy class of $\ga$. The product is given by $(a,[\ga_{1}],b)(b,[\ga_{2}],c)=(a,[\ga_{1}\circ \ga_{2}],c)$ and inversion by 
$(a,[\ga_{1}],b)^{-1}=(b,[\ga_{1}^{-1}],a)$.  Assume that $G$ is Hausdorff.  Then $G$ is a continuous family groupoid.  To check this, conditions (i) and (ii) of 
Definition~\ref{def:sfg} are obvious.  For (iii), we just have to show that locally, the map
$(x,[\ga],y)\to f_{(x,[\ga],y)}$, where $f_{(x,[\ga],y)}(y,[\de],z)=(x,[\ga\circ\de],z)$ is 
$C^{\infty,0}$.  This follows using Connes's coordinatizing of the charts $W_{[\ga]}(U,V)$ for $G$ (e.g. \cite[p.69]{Paterson}).\\

(4) {\bf The $G$-invariant classical elliptic complexes}\\

Let $G, X$ be as in \S 5 of the paper.  Using a $C^{\infty,0}$-hermitian metric on $TX$ that is $G$-isometric, the classical elliptic complexes have their natural $G$-invariant versions.  Consider, for example, the de Rham complex $\{d_{i}\}$ on $X$.   Here, $d_{i}:C^{\infty,0}(E^{i})\to 
C^{\infty,0}(E^{i+1})$, where $E^{i}$ is the natural $C^{\infty,0}$ bundle whose fiber at $x$ is 
$\La^{i}(TX_{x}^{*})\otimes \C$, and $d_{i}$ is the exterior derivative.  Introduce $G$-invariant hermitian metrics on the $E^{i}$ (Proposition~\ref{prop:hermG}). Then as in \cite[p.521]{AS1}, $D:C^{\infty,0}(\oplus_{i}E^{2i})\to C^{\infty,0}(\oplus_{i}E^{2i+1})$, where 
$D=d+d^{*}$, is an elliptic, invariant, differential family on $X$ to which the theory of the present paper applies.  This is the {\em $G$-invariant de Rham family}.  Similarly, if $X$ is a continuous family of complex manifolds and $G$ acts holomorphically, then there is a $G$-invariant Dolbeault elliptic family. 

Next, suppose that $X$ is a $C^{\infty,0}$-oriented family.  This can be defined by adapting any of the usual definitions of orientation: for example, that there exists a non-vanishing $C^{\infty,0}$ longitudinal $k$-form on $X$ (where $k$ is the dimension of the $X^{u}$'s).   The $X^{u}$'s are then orientable manifolds.  Suppose further that the $G$-action preserves the orientation, that the $X^{u}$'s are all compact and that $k$ is divisible by $4$.  Then 
there is defined a $G$-invariant, elliptic differential family $D^{+}$ of Hodge signature operators for $X$.  Of course each $(D^{+})^{u}$ is the usual Hodge signature operator on $X^{u}$.  (The existence of this family is noted in \cite[p.134]{AS4} for the non-equivariant case and with $X$ a continuous family in the sense of Atiyah and Singer.)  In the case of a compact Lie group action, one has (\cite[p.578]{AS3}, \cite[p.140f.]{Shanahan}) 
$\text{ind}(D^{+})=\rho^{+} - \rho^{-}\in R(G)=K(C^{*}(G))$.  The $G$-Signature Theorem 
computes the character of this K-class.

Lastly, there are natural $C^{\infty,0}$ versions of spin manifolds and spin actions for $X$ and 
$G$.  (These are similar to the compact group versions - see, for example, \cite[p.158ff.]{Shanahan}.)
In those circumstances, and with $k$ even, there is a $G$-invariant Dirac elliptic family $D$.  
A special case of this, where $G$ is a bundle of Lie groups acting properly and smoothly on a fiber bundle $Y$, arises in the analysis of the Dirac operator on certain non-compact manifolds, and has been investigated in detail by Nistor (\cite{Nistor}).  He obtains the equivariant analytic index of $D$ in $K_{0}(C^{*}_{red}(G))$ by means of an element of $K_{0}(C^{*}_{r}(Y;G))$ (where $C^{*}_{r}(Y;G)$ is a ``regularizing'' algebra).  The theorem of the present paper also applies to that situation.  Local index theorems using traces are given by Nistor in \cite{Nistoras} when $G$ is a vector bundle.\\   

\noindent
(5) {\bf Hypoelliptic $L^{m}_{\rho,\de}$-pdo's and the Baum-Connes conjecture}\\

The reader may wonder why the pdo's of this paper are taken to be $L^{m}_{\rho,\de}$ instead of classical.  The justification for this is that 
index theory extends to elliptic, and indeed to hypoelliptic, $L^{m}_{\rho,\de}$-pdo's and (as will be seen below) such operators are required in noncommutative geometry.  Index theory for these pdo's was developed by H\"{o}rmander in \cite{Horindex}.  The definition of hypoelliptic (\cite[p.38]{Shubin}) is the same as that of elliptic except that we require there to exist an $m_{0}\in \R$ for which (\ref{eq:ellpdo}) is replaced by:
\begin{equation}           \label{eq:ellhypo}
C_{1}\left|\xi\right|^{m_{0}}\leq \left|\si_{D'}(x,\xi)\right|\leq C_{2}\left|\xi\right|^{m}.
\end{equation}
To prove the coincidence of the analytic and topological indices for a hypoelliptic pdo, he showed that it is connected up to elliptic operators and hence, by the continuity of the index for continuous families of pdo's, the indices coincide.  In the final part of the paper, he 
uses a continuous family of hypoelliptic $L^{0}_{1,1/2}$-pdo's to determine analytically the index of the Bott element. 

In noncommutative geometry, the classical elliptic calculus is not enough.  In particular, a hypoelliptic, $L^{m}_{\rho,\de}$ calculus is required in the study of foliations.  This calculus is developed in the paper of Hilsum and Skandalis (\cite{HilSkand}), where a K-homology analogue of a theorem of Connes ((\cite[Theorem 6.8]{ Cocyclic}) is proved.  Hilsum and Skandalis also prove that if $(V_{1},F_{1}), (V_{2},F_{2})$ are foliations, then every K-oriented morphism $f:V_{1}/F_{1}\to V_{2}/F_{2}$ determines a class $f!\in 
KK^{*}(C^{*}(V_{1}, F_{1}), C^{*}(V_{2}, F_{2}))$. (The existence of $f!$ was conjectured by Connes in (\cite{Cofol}).)  The case where $f$ is a submersion uses the hypoelliptic calculus.  Hilsum and Skandalis illustrate helpfully the difficulty for this case by considering the submersion $p:V/F\to \text{pt}$ ($(V,F)$ a foliation).  Replace the holonomy groupoid of $(V,F)$ by an \'{e}tale version, which, for the purpose of illustration, is taken to be a transformation group $X\x_{\al}\Ga$, where $\Ga$ is a group of diffeomorphisms on $X$.  (See \cite[p.139]{Cocyclic}.)  The natural candidate for $p!$ should then be determined by a Dirac operator $D$ on $X$.  Unfortunately, $\Ga$ may not preserve a Riemannian metric on $X$ and as a consequence, $D$ may not satisfy all of the required properties to give a Kasparov bimodule.  To solve this difficulty, the authors use an idea of Connes (\cite{Cocyclic}).  There is an ``almost isometric'' structure present, and $D$ is to be replaced by the sum of two ``partial'' Dirac operators differentiating in ``complementary'' directions.  This operator is hypoelliptic and of type $(\rho,\de)$ ($\de=1-\rho$).  For the general situation, one has to consider {\em tranversally elliptic} pseudodifferential operators, and, in a long appendix, Hilsum and Skandalis develop the theory for such operators in relation to connections and Kasparov products. 

Connes and Moscovici (\cite{CoMoscindex}), in their work on the local index formula in noncommutative geometry, refined the construction of \cite{HilSkand} so that it applies in the context of the spectral triple of a triangular structure.  They constructed a hypoelliptic operator by combining a longitudinal signature operator of order two with the usual signature operator in the transverse direction.  This is used in their study of the problem of computing, by a local formula, the cyclic cohomology Chern character of a spectral triple. The computation adapts the {\em Wodzicki residue}, the unique extension of the Dixmier trace to pseudodifferential operators.  (The problem of the computation of the local index for transversally hypoelliptic operators on foliations was later solved by Connes and Moscovici (\cite{CoMoschopf}).)  The pseudodifferential calculus used in \cite{CoMoscindex} is a special case of the pseudodifferential calculus on Heisenberg manifolds developed by Beals and Greiner (\cite{BG}).  Developing this theme, Ponge (\cite{Pongecal} and in his Ph.D. thesis) has constructed a sub-elliptic functional calculus for Heisenberg manifolds, and a noncommutative residue.

The work of J.L. Tu (\cite{Tunov,Tumoy,Tutrees,Tupre}) investigates the Baum-Connes conjecture for groupoids, using the groupoid equivariant KK-theory of P.-Y. Le Gall (\cite{LeGall0,LeGall,LeGall1}) and the Dirac-dual Dirac method (\cite[Chapter 9]{Valette}, \cite{Cocyclic}).  Tu determines (\cite[Th\'{e}or\`{e}me 5.24]{Tunov}) general conditions on a locally compact groupoid $G$ and the classifying space for proper $G$-actions
which ensure that the Baum-Connes map is injective.  This is used to show that a large class of foliations satisfy the Novikov conjecture.  The Baum-Connes map constructed by Tu in \cite{Tunov} applies to continuous family groupoids.  Further, the construction extends, to the groupoid context, part of the argument for the group equivariant index theorem given in outline much earlier by Kasparov in \cite{Kasp1}.  It seems to me hopeful that the rest of Kasparov's argument, using {\em the generalized Atiyah-Singer theorem} - in which the K-homology class of $D$ is shown to be a certain intersection product - will extend to the continuous family groupoid context, and that the conjecture of Connes, discussed in the introduction, can be established in complete generality.


\end{document}